\documentclass[11pt,reqno]{amsart}
\usepackage{enumerate}
\usepackage{enumitem}
\usepackage{mathrsfs}
\usepackage{url}
\usepackage{comment}
\usepackage{amsfonts, amssymb, amscd, amsthm, bm, cancel}
\usepackage{amsmath}
\usepackage{hyperref}
\usepackage[margin=1.3in]{geometry}
\hypersetup{
	linktocpage=true,
	colorlinks=true,
	citecolor=magenta,
	linkcolor=blue,
	urlcolor=magenta
}
\providecommand{\arxiv}[1]{\href{https://arxiv.org/abs/#1}{arXiv:#1}}

\newtheorem{proposition}{Proposition}[section]
\newtheorem{theorem}[proposition]{Theorem}
\newtheorem{lemma}[proposition]{Lemma}

\theoremstyle{definition}
\newtheorem{definition}[proposition]{Definition}

\theoremstyle{remark}

\numberwithin{equation}{section}
\allowdisplaybreaks

\begin{document}
\title[Higher regularity of the inverse anisotropic mean curvature flow]{Higher regularity of the inverse anisotropic mean curvature flow}

\author[C. Gao]{Chaoqun Gao}
\author[Y. Wei]{Yong Wei}
\author[R. Zhou]{Rong Zhou}

\address{School of Mathematical Sciences, University of Science and Technology of China, Hefei 230026, P.R. China}
\email{\href{mailto:gaochaoqun@mail.ustc.edu.cn}{gaochaoqun@mail.ustc.edu.cn}}
\email{\href{mailto:yongwei@ustc.edu.cn}{yongwei@ustc.edu.cn}}
\email{\href{mailto:zhourong@mail.ustc.edu.cn}{zhourong@mail.ustc.edu.cn}}

\date{\today}
	
		\begin{abstract}
		We prove an anisotropic analogue of the higher regularity theorem of Huisken and Ilmanen for inverse mean curvature flow. For an arbitrary smooth Minkowski norm, we first prove a Huisken--Ilmanen type Harnack estimate for smooth closed strictly star-shaped solutions. We then construct global smooth solutions starting from $C^1$ strictly star-shaped hypersurfaces with bounded nonnegative weak anisotropic mean curvature. Combining this construction with the asymptotic theory for weak inverse anisotropic mean curvature flow, we show that weak solutions starting from bounded smooth initial sets become smooth outside a compact set.
	\end{abstract}
	
	\maketitle
	\tableofcontents
	
	
	\section{Introduction}
	\label{section1}
	Inverse mean curvature flow (IMCF) is the expanding geometric flow of hypersurfaces whose normal speed is the reciprocal of the mean curvature. In Euclidean space, a classical solution is a smooth family $X:M^n\times[0,T]\to \mathbb{R}^{n+1}$ ($n\geq 2$) satisfying
	\begin{equation}
		\label{ICMF}
		\frac{\partial }{\partial t}X(x,t)=\frac{1}{H(x,t)}\nu(x,t), 
	\end{equation}
	where $H>0$ is the mean curvature and $\nu$ is the outward unit normal of $M_t=X(M,t)$. Gerhardt \cite{G90} and Urbas \cite{U90} proved that \eqref{ICMF} has a global smooth solution when the initial hypersurface is smooth, star-shaped, and strictly mean convex. After suitable rescaling, the solution converges to a round sphere.
	
	The weak theory of IMCF developed by Huisken and Ilmanen \cite{HI01} allows one to start the flow from arbitrary smooth compact sets and to permit jumps of the level sets. This theory was a central ingredient in their proof of the Riemannian Penrose inequality. The regularity of such weak solutions is a subtle question because the equation degenerates where $H$ approaches zero. In their higher regularity paper \cite{HI08}, Huisken and Ilmanen proved that strict star-shapedness compensates for this degeneracy and yields smooth solutions from weakly mean convex $C^1$ data. We recall the form of their theorem that motivates the present work.
	
	\begin{theorem}[see {\cite[Theorem 2.5]{HI08}}]
		\label{higher-HI}
		Let $X_0:M^n\to\mathbb{R}^{n+1}$ be a closed embedded hypersurface of class $C^1$ with measurable, bounded, nonnegative weak mean curvature $H\geq 0$. Assume that $M_0=X_0(M^n)$ is strictly star-shaped, namely
		\begin{equation*}
			0<R_1\leq \langle X,\nu\rangle\leq R_2.
		\end{equation*}
		Then IMCF has a global smooth solution $X:M^n\times(0,\infty)\to\mathbb{R}^{n+1}$. As $t\downarrow 0$, $M_t=X(M^n,t)$ converges to $M_0$ uniformly in $C^0$. 
	\end{theorem}
	
	The proof in \cite{HI08} has two main ingredients. The first is a sharp lower bound for the mean curvature of star-shaped solutions, or equivalently a Harnack estimate for the speed. The second is an approximation lemma that replaces a weakly mean convex $C^1$ initial hypersurface by smooth strictly mean convex hypersurfaces. These two ingredients give global smooth solutions from rough data. Combined with the asymptotic roundness of the level sets, they imply that weak IMCF becomes smooth after the first strictly star-shaped time \cite[Theorem 2.7]{HI08}.
	
	Related regularity results have since been obtained in other ambient spaces. Li and the second author \cite{LW17} studied IMCF in Kottler spaces for star-shaped weakly mean convex initial hypersurfaces. Shi and Zhu \cite{SZ21} proved eventual smoothness and star-shapedness for weak IMCF in three-dimensional asymptotically hyperbolic manifolds, and this was extended to higher dimensions $3\leq n\leq 7$ by Harvie \cite{Har24} for weak solutions in hyperbolic space. Harvie and Wang \cite{HarW25} also proved the analogous regularity results for the weak IMCF in asymptotically flat manifolds.
	
	The purpose of this paper is to prove an anisotropic counterpart of the higher regularity theorem of Huisken and Ilmanen \cite{HI08}. Let $F\in C^{\infty}(\mathbb{R}^{n+1}\setminus\{0\})$ be a Minkowski norm, so that $F$ is a norm and $D^2(\frac{1}{2}F^2)$ is positive definite on $\mathbb{R}^{n+1}\setminus\{0\}$. The associated inverse anisotropic mean curvature flow (IAMCF) is
	\begin{equation}
		\label{IAMCF}
		\frac{\partial }{\partial t} X(x,t)=\frac{1}{H_F(x,t)}\nu_F(x,t),
	\end{equation}
	where $H_F$ is the anisotropic mean curvature and $\nu_F$ is the anisotropic normal of $M_t=X(M,t)$. When the initial hypersurface is smooth, star-shaped, and strictly $F$-mean convex, Xia \cite{Xia17} proved the long-time existence of the flow \eqref{IAMCF} and convergence, after rescaling, to a rescaled Wulff shape.
	
	Weak IAMCF can be formulated through the anisotropic level set equation
	\begin{equation}
	\label{eq:WIAMCF}
	\begin{cases}
		\operatorname{div}(DF(Du))=F(Du) & \text{in } \mathbb{R}^{n+1}\setminus\Omega, \\
		u=0 & \text{on } \partial\Omega, \\
		u\to\infty & \text{as } |x|\to\infty.
	\end{cases}
	\end{equation}
	If $E_t=\{x\in\mathbb{R}^{n+1}:u(x)<t\}$ and $M_t=\partial E_t$, then the level sets of a solution to \eqref{eq:WIAMCF} represent weak IAMCF. Della Pietra, Gavitone and Xia \cite{DGX23} established existence and uniqueness of proper weak solutions in this setting.
	
	\begin{theorem}[see {\cite[Theorem 1.1]{DGX23}}]
		\label{thmwiamcf}
		Let $\Omega\subset\mathbb{R}^{n+1}$ be an open bounded set with smooth boundary. Then there exists a unique proper weak solution $u\in C^{0,1}_{\mathrm{loc}}(\mathbb{R}^{n+1}\setminus \Omega)$ of \eqref{eq:WIAMCF} in the sense of Subsection \ref{sec.2-wiamcf} with $u=0$ on $\partial\Omega$. 
Moreover,
	\begin{equation*}
			F(Du(x))\leq \sup_{\partial\Omega}H_F^+, \qquad x\in \mathbb R^{n+1}\setminus\overline\Omega,
		\end{equation*}
		and
		\begin{equation*}
			F(Du(x))\leq H_F^+(x), \qquad x\in\partial\Omega,
		\end{equation*}
		where $H_F^+=\max\{H_F,0\}$ and $H_F$ is the anisotropic mean curvature of $\partial\Omega$.
	\end{theorem}
	
	Cabezas-Rivas, Moll and Solera \cite{CMS24} later obtained existence of weak IAMCF under weaker assumptions on both the anisotropy and the initial set. The authors \cite{GWZ26} studied the asymptotic behavior of weak IAMCF and showed that the weak solution is asymptotic to the expanding Wulff shape solution. These results lead to the following question: do weak solutions of IAMCF become smooth after sufficiently large time, as in the isotropic theory of Huisken and Ilmanen \cite{HI08}? We answer this question for an arbitrary smooth Minkowski norm.

	\subsection{Huisken--Ilmanen type Harnack estimate}
	\label{subsec:harnack-intro}
We first prove a Huisken--Ilmanen type Harnack estimate for smooth star-shaped solutions of \eqref{IAMCF}. It holds for an arbitrary smooth Minkowski norm $F$.
	
	For a hypersurface with Euclidean outward normal $\nu$, we write
	\begin{equation*}
		\chi=\frac{\langle X,\nu\rangle}{F(\nu)}
	\end{equation*}
	for its anisotropic support function. Thus strict anisotropic star-shapedness means $\chi>0$.
	
	\begin{theorem}
		\label{prop-harnack}
		Let $F$ be a Minkowski norm, and $X:M^n\times[0,T]\to\mathbb{R}^{n+1}$ be a smooth closed solution to \eqref{IAMCF} with $H_F>0$. Assume that the initial hypersurface $M_0=X(M,0)$ satisfies
		\begin{equation*}
			0< R_1 \leq \chi \leq R_2
		\end{equation*}
		for some constants $R_1,R_2>0$. Then there is a constant $C=C(n,F)>0$ such that
		\begin{equation}
			\label{harnackest}
			\dfrac{1}{H_F} \leq C(n,F)\left( \dfrac{R_2}{R_1} \right)^{\frac{3}{2}}\left( 1 + \dfrac{1}{t^{\frac{1}{2}}} \right)R_2\mathrm{e}^{\frac{t}{n}}
		\end{equation}
		holds everywhere on $M\times(0,T]$.
	\end{theorem}

	A key point of the estimate \eqref{harnackest} is that it  depends only on $R_1$, $R_2$, the dimension, and the fixed anisotropy, and does not depend on the initial anisotropic mean curvature. 
 
	The proof is based on a pointwise maximum principle argument. Inspired by the alternative proof of the Huisken--Ilmanen estimate given by Choi and Daskalopoulos \cite[Theorem A.5]{CD21}, we consider the auxiliary quantity
	\begin{equation*}
		W = \frac{\varphi^{1-\varepsilon}(w)\,\mathrm{e}^{\gamma G(\nu_F)(X,X)}}{H_F},
	\end{equation*}
	where $w=\chi^{-1}$, $G(\nu_F)(X,X)$ is the anisotropic analogue of $|X|^2$, and $\varphi$, $\varepsilon$, and $\gamma$ are chosen appropriately. Applying the maximum principle to $tW$ on $M\times[0,T]$ gives a bound for the speed $1/H_F$. This differs from the original proof of Huisken and Ilmanen \cite{HI08}, which combines the evolution of the support function and the mean curvature with the Michael--Simon Sobolev inequality and a Stampacchia iteration. In the anisotropic setting, the evolution equations contain extra terms involving the tensors $Q$ and $T$ associated with the dual norm, so that the Huisken--Ilmanen argument based on Sobolev inequality and iteration does not transfer directly. The maximum principle method allows these terms to be absorbed into constants depending on $F$, without imposing any additional assumptions on the anisotropy.

	\subsection{Higher regularity of IAMCF}
	\label{subsec:higher-regularity-intro}
	The Harnack estimate gives a lower bound for $H_F$ along smooth star-shaped solutions, but to start the flow from rough weakly $F$-mean convex data one also needs an approximation theorem. We obtain this approximation by running anisotropic mean curvature flow for a short time. 
    

	The smooth existence theorem below is the anisotropic counterpart of Theorem \ref{higher-HI}.
	
	\begin{theorem}
	\label{higher-our}
		Let $F$ be a Minkowski norm. Let $M_0=X_0(M)$ be a closed embedded hypersurface of class $C^1$ with measurable, bounded, nonnegative weak anisotropic mean curvature $H_F\geq 0$. 
		Assume that $X_0$ is strictly star-shaped in the anisotropic sense that
		\begin{equation*}
			0<R_1\leq \chi=\frac{\langle X,\nu\rangle}{F(\nu)}\leq R_2
		\end{equation*}
		for some constants $R_1,R_2>0$. Then IAMCF \eqref{IAMCF} has a global smooth solution $X:M^n\times(0,\infty)\to\mathbb{R}^{n+1}$. As $t\downarrow 0$, $M_t=X(M^n,t)$ converges to $M_0$ uniformly in $C^0$.
	\end{theorem}
	
	The proof follows the two-step scheme of Huisken and Ilmanen \cite{HI08}, but both steps require new anisotropic ideas. The first step is the Harnack estimate of Subsection \ref{subsec:harnack-intro}. Instead of the global integral argument in \cite{HI08}, we use the auxiliary-function maximum principle described above. This pointwise method is well suited to the anisotropic evolution equations, whose lower-order terms involve the tensors $Q$ and $T$ coming from derivatives of the dual norm.
	
	The second step is the approximation of the initial hypersurface. We first use elliptic regularity and radial mollification to obtain smooth approximants with a common $C^{1,\beta}$ bound. We then run anisotropic mean curvature flow for a short time. A stopping-time argument keeps the local graphical equations uniformly parabolic, while initial-time barriers, H\"older-gradient estimates, interpolation and parabolic Schauder estimates yield uniform $C^1$ convergence and an integrable curvature bound. The strong maximum principle gives strict $F$-mean convexity of the limiting approximating hypersurfaces. Starting the smooth IAMCF from these approximations and using the Harnack estimate yields uniform estimates on compact time intervals away from $t=0$, and a compactness argument gives the desired smooth solution.
	
	Combining Theorem \ref{higher-our} with the asymptotic behavior of the weak IAMCF yields our main regularity theorem for weak solutions starting from bounded smooth sets, which provides an anisotropic counterpart of Huisken--Ilmanen's \cite[Theorem 2.7]{HI08}.
	
	\begin{theorem}
		\label{thm-main}
		Let $F$ be a Minkowski norm. Let $E_0=\Omega\subset\mathbb{R}^{n+1}$ be a bounded open set with smooth boundary, and let $u$ be the proper weak solution of IAMCF \eqref{eq:WIAMCF} starting from $\Omega$. Set $E_t=\{x\in\mathbb{R}^{n+1}:u(x)<t\}$ and $M_t^n=\partial E_t$. Then there exists $t_0\geq 0$ such that $M_{t_0}^n$ is strictly star-shaped and of class $C^1$. Moreover, $M_t^n$ is smooth for every $t>t_0$. Equivalently, with $K=\overline{E_{t_0}}$, all level hypersurfaces are smooth in $\mathbb{R}^{n+1}\setminus K$.
	\end{theorem}
	
	The proof uses the asymptotic behavior of weak IAMCF from our previous work \cite{GWZ26}. Together with the weak gradient estimate, compactness for anisotropic almost minimizers, and small-excess regularity, it gives $C^{1,\beta}$ convergence of the rescaled level sets to a Wulff shape. Hence some large level set is strictly star-shaped and of class $C^1$, with nonnegative bounded weak anisotropic mean curvature. Theorem \ref{higher-our} then gives a smooth classical continuation. Uniqueness in the weak formulation identifies this continuation with the original weak flow.

   The paper is organized as follows. Section \ref{sec2} recalls the anisotropic geometry and the weak formulation of IAMCF. Section \ref{sec3} proves the Harnack estimate. Section \ref{sec4} constructs global smooth solutions from $C^1$ weakly $F$-mean convex, strictly star-shaped initial hypersurfaces. Section \ref{sec5} proves eventual smoothness for weak solutions.

\section{Preliminaries}
	\label{sec2}
	In this section, we collect some preliminaries on anisotropic geometry and the weak formulation of inverse anisotropic mean curvature flow.
	
	\subsection{Minkowski norm and Wulff shape}\label{sec.2-1}
	We recall the definitions of the Minkowski norm and the Wulff shape.
	
	A function $F\in C^{\infty}\left(\mathbb{R}^{n+1}\backslash\{0\}\right)$ is called a \textit{Minkowski norm} if
    \begin{enumerate}
        \item $F$ is a norm in $\mathbb{R}^{n+1}$, i.e., $F$ is convex, even, $1$-homogeneous with $F(\xi)\geq 0$ for all $\xi\in \mathbb{R}^{n+1}$, and $F(\xi)=0$ if and only if $\xi=0$.
        \item $F$ satisfies a uniformly elliptic condition: $D^2\left( \frac{1}{2}F^2 \right)$ is positive definite in $\mathbb{R}^{n+1}\backslash \{0\}$.
    \end{enumerate}
    The dual norm of a Minkowski norm  $F$ is defined as
	\begin{equation}
		F^{\circ}(x):= \sup\limits_{\xi\neq 0} \dfrac{\langle x,\xi \rangle}{F(\xi)},\ \ x\in\mathbb{R}^{n+1},
		\label{F0}
	\end{equation}
	and $F^{\circ}$ is also a Minkowski norm in $C^{\infty}\left( \mathbb{R}^{n+1}\backslash \{0\} \right)$. From the definition \eqref{F0}, we have the anisotropic Cauchy--Schwarz inequality:
	\begin{equation}
		\langle x,\xi \rangle\leq F^{\circ}(x)F(\xi),\ \ \forall\ x,\ \xi\in\mathbb{R}^{n+1}
		\label{anisocauchy}
	\end{equation}
	If $x\neq0$ and $\xi\neq0$, equality holds if and only if $x=F^{\circ}(x)DF(\xi)$ and $\xi=F(\xi)DF^{\circ}(x)$. The cases in which one vector vanishes are immediate.
	
	The following properties of $F$ and $F^{\circ}$ hold true, see e.g. \cite{BC18,CS09,Xia12}.
	\begin{lemma}
		For any $x,\xi\in\mathbb{R}^{n+1}\backslash \{0\}$, we have
		\begin{align}
			\langle DF(\xi),\xi \rangle=F(\xi),&\ \ \langle DF^{\circ}(x),x \rangle=F^{\circ}(x). \label{fprop1}\\
			F\left( DF^{\circ} (x)\right)=1,&\ \ F^{\circ} \left( DF(\xi) \right) = 1, \\
			F^{\circ}(x) DF\left( DF^{\circ}(x) \right) = x,&\ \  F(\xi) DF^{\circ} \left( DF(\xi) \right) = \xi.\label{fprop3}
		\end{align}
		Here $D$ denotes the standard gradient operator in $\mathbb{R}^{n+1}$.
	\end{lemma}
	
	Given a Minkowski norm $F$ in $\mathbb{R}^{n+1}$, the associated Wulff shape is defined as
	\begin{equation*}
		\mathcal{W}:=\left\lbrace x\in\mathbb{R}^{n+1}:F^\circ(x)<1 \right\rbrace.
	\end{equation*}
	For $x_0\in \mathbb{R}^{n+1}$ and $r>0$, we denote by $\mathcal{W}_r(x_0)=r\mathcal{W}+x_0$ the scaled and translated Wulff shape, which satisfies
	\begin{equation*}
		\mathcal{W}_r(x_0)=\left\lbrace x\in\mathbb{R}^{n+1}:F^\circ(x-x_0)<r \right\rbrace.
	\end{equation*}
	When $F$ is the Euclidean norm, $\mathcal{W}_r(x_0)$ coincides with the Euclidean ball $\mathbb{B}_r(x_0)$ of radius $r$ centered at $x_0$.

	\subsection{Anisotropic curvatures for hypersurfaces}
	We recall the anisotropic geometry of hypersurfaces. More details can be found in \cite{Andrews01}, \cite{Xia12}, and \cite{Xia17}.
	
	The Riemannian metric $G$ with respect to $F^{\circ}$ on $\mathbb{R}^{n+1}$ is defined by
	\begin{equation*}
		G(x)(V,W):=\sum\limits_{\alpha,\beta=1}^{n+1} \dfrac{\partial^2 \left( \frac{1}{2}\left(F^{\circ}\right)^2(x) \right)}{\partial x^{\alpha}\partial x^{\beta}} V^{\alpha}W^{\beta},\ \ \forall\ x\in\mathbb{R}^{n+1}\backslash\{0\},\ V,W\in T_{x}\mathbb{R}^{n+1}.
	\end{equation*}
	Since the third and fourth derivatives of $F^{\circ}$ do not generally vanish, for all $x\in\mathbb{R}^{n+1}\backslash\{0\}$ and $U,V,W,Z\in T_{x}\mathbb{R}^{n+1}$, we denote
	\begin{equation*}
		Q(x)(U,V,W):= \sum\limits_{\alpha,\beta,\gamma=1}^{n+1} \dfrac{\partial^3 \left( \frac{1}{2}\left(F^{\circ}\right)^2(x) \right)}{\partial x^{\alpha}\partial x^{\beta} \partial x^{\gamma}} U^{\alpha}V^{\beta}W^{\gamma},
	\end{equation*}
	and
	\begin{equation*}
		T(x)(U,V,W,Z) := \sum\limits_{\alpha,\beta,\gamma,\delta=1}^{n+1}\dfrac{\partial^4\left( \frac{1}{2}\left( F^{\circ} \right)^2(x) \right)}{\partial x^{\alpha}\partial x^{\beta}\partial x^{\gamma}\partial x^{\delta}}U^{\alpha}V^{\beta}W^{\gamma}Z^{\delta}.
	\end{equation*}
	
	Let $X:M\to\mathbb{R}^{n+1}$ be a hypersurface, and let $\nu$ denote its unit outward normal. The anisotropic Gauss map  is a map $\nu_F:M \to \partial\mathcal{W}$ defined by
	\begin{equation*}
    \nu_F(p):=DF(\nu(p))=F(\nu(p))\nu(p)+\nabla_{\mathbb{S}^n}F(\nu(p)).
	\end{equation*}
	The differential gives the anisotropic Weingarten map 
    \begin{equation*}
		\mathrm{d}\nu_F : T_{p}M \to T_{\nu_F(p)}\partial\mathcal{W},
	\end{equation*}
whose eigenvalues are called the  anisotropic principal curvatures $\kappa^{F}=\left( \kappa_1^F,\kappa_2^F,\cdots,\kappa_n^F \right)$.	
	Note that $\nu_F \in \partial\mathcal{W}$ for $p\in M$, by \eqref{fprop1} and \eqref{fprop3} we have
	\begin{equation*}
		G(\nu_F)(\nu_F,\nu_F)=1,\ \ G(\nu_F)(\nu_F,V)=0,\ \ \text{for}\ V\in T_{p}M,
	\end{equation*}
	\begin{equation}
		Q(\nu_F)(\nu_F,V,W)=0,\ \ \text{for}\ V,W\in\mathbb{R}^{n+1}.
		\label{equ-q}
	\end{equation}
	Then the induced Riemannian metric on $M\subset\mathbb{R}^{n+1}$ with respect to $G$ is given by
	\begin{equation*}
		g(p):= G\left(\nu_F(p)\right)|_{T_p M},\ p\in M.
	\end{equation*}
	Denote the covariant derivatives of $g$ and $G$ by $\nabla$ and $\hat{D}$ respectively, then the first and second fundamental form of $(M,g)\subset (\mathbb{R}^{n+1},G)$ are
	\begin{equation*}
		g_{ij}:= G(\nu_{F})(\partial_i X,\partial_j X),\ \ \qquad h_{ij}:= G(\nu_F)(\hat{D}_{\partial_i}\nu_F,\partial_j X).
	\end{equation*}
	In these notations,  the anisotropic principal curvatures are eigenvalues of 
    \begin{equation*}
      \left(h^i_j\right)=\left(g^{ik}h_{kj}\right),
    \end{equation*}
and the anisotropic mean curvature with respect to $F$ is the trace of $h$: 
	\begin{equation*}
		H_F = \mathrm{tr}_{g}(h) = g^{ij}h_{ij}.
	\end{equation*}
	We say that $M$ is strictly $F$-mean convex if $H_F>0$, and $F$-mean convex if $H_F\geq 0$.
	
	We have the following anisotropic Gauss--Weingarten type formulas and the anisotropic Gauss-Codazzi equation.
	\begin{lemma}[see {\cite[Lemma 6.7]{Xia12}}]
		\label{lem-formula}
		\begin{align}
			\partial_i\partial_j X = & ~- h_{ij}\nu_F + \nabla_{\partial_i} \partial_j X+ g^{k\ell} A_{ij\ell}\partial_k X \quad &\text{(Gauss formula)} \label{gaussformula}\\
			\partial_i \nu_F = & ~g^{jk}h_{ij}\partial_k X \quad &\text{(Weingarten formula)}   \label{weinformula}\\
			R_{ijk\ell} =&~ h_{ik}h_{j\ell} - h_{i\ell}h_{jk} + \nabla_{\ell}A_{jki}-\nabla_{k}A_{j\ell i}\notag  \\
			&\quad  + g^{pm}A_{jkp}A_{m\ell i} - g^{pm}A_{j\ell p}A_{mki} \notag \\
			=& ~h_{ik}h_{j\ell} - h_{i\ell}h_{jk} + D_{ijk\ell} \quad &\text{(Gauss equation)}  \notag\\
			\nabla_k h_{ij} =& ~\nabla_{i} h_{jk}  + h_k^{\ell} A_{\ell i j} -h_{i}^{\ell} A_{\ell k j} \notag \\
			=&~ \nabla_{i} h_{jk} + C_{ijk} \quad &\text{(Codazzi equation)}  \label{codazziequ}
		\end{align}
		Here $R$ is the Riemannian curvature tensor of $g$ defined by
		\begin{equation*}
			R_{ijkl} =g\left( \nabla_{k}\nabla_{\ell} \partial_j - \nabla_{\ell}\nabla_{k} \partial_j, \partial_i \right),
		\end{equation*}
		and $A , C$   are (0,3) tensors
		\begin{align}
			A_{ijk} = &~-\dfrac{1}{2}\left( h^{\ell}_i Q_{jk\ell} + h^{\ell}_j Q_{i\ell k} - h^{\ell}_k Q_{ij\ell} \right),  \label{Atensor} \\
			C_{ijk}= &~h_k^{\ell} A_{\ell i j} -h_{i}^{\ell} A_{\ell k j},\nonumber
		\end{align}
		and $D$ is a (0,4) tensor
		\begin{equation*}
			D_{ijk\ell}= \nabla_{\ell}A_{jki}-\nabla_{k}A_{j\ell i}\notag
			+ g^{pm}A_{jkp}A_{m\ell i} - g^{pm}A_{j\ell p}A_{mki},
		\end{equation*}
		where $Q_{ijk}=Q(\nu_F)(\partial_i X, \partial_j X, \partial_k X)$.
	\end{lemma}

	Let $\mathrm{d}\mu$ be the area element of the hypersurface $M$ in ($\mathbb{R}^{n+1},\delta)$ with respect to the induced metric from the Euclidean metric, the anisotropic area element of $M$ with respect to $F$ is defined by
	\begin{equation*}
		\mathrm{d}\mu_F = F(\nu)\mathrm{d}\mu.
	\end{equation*}
	For a smooth hypersurface $M$ we write
	\begin{equation*}
		|M|_F:=\int_M\mathrm{d}\mu_F,
	\end{equation*}
	and, for a set $E$ of locally finite perimeter and an open set $U$,
	\begin{equation*}
		P_F(E;U):=\int_{\partial^*E\cap U}F(\nu_E)\,\mathrm{d}\mathcal H^n.
	\end{equation*}
	\begin{lemma}[see {\cite[Lemma 2.8]{Xia12}}]
		Let $\mathrm{d}\mu_{g}$ be the induced metric volume form of $(M,g)$. Assume that
		\begin{equation*}
			\mathrm{d}\mu_F = F(\nu)\mathrm{d}\mu = \varphi\mathrm{d}\mu_{g},
		\end{equation*}
		then $\nabla_{i} \log\varphi = g^{jk}A_{ijk}$.  Consequently, for any two functions $f_1, f_2\in C^{\infty}(M)$, we have the integration by parts:
		\begin{equation}\label{intebyparts}
			\int_{M} f_1\left(\Delta f_2+ g^{jk}A_{ijk}\nabla^i f_2\right) \mathrm{d}\mu_F = -\int_M\langle \nabla f_1,\nabla f_2\rangle_{g}d\mu_F.
		\end{equation}
	\end{lemma}
	
	The anisotropic support function of $M$ with respect to $F$ is defined as \begin{equation*}
		\chi=G(\nu_F)(\nu_F,X).
	\end{equation*}
	Using \eqref{fprop1}, \eqref{fprop3} and \eqref{anisocauchy}, we see that
	\begin{equation}
		\chi=G(\nu_F)(\nu_F,X)=\langle DF^{\circ}(DF(\nu)),X \rangle=\dfrac{\langle X,\nu \rangle}{F(\nu)}\leq F^{\circ}(X).
		\label{equ-uf0}
	\end{equation}
	
	If $M=\partial\mathcal{W}$, then $\nu_F(X)=X$, $\chi=1$, $h_{ij}=g_{ij}$ and $H_F=n$.
	
	\subsection{Anisotropic mean curvature flow}
	\label{appB}
	We recall some basic results on anisotropic mean curvature flow, which will be used in Subsection \ref{sec.4-1} for smoothing a $C^1$ hypersurface. More details can be found in   \cite{Andrews01}, \cite{WX21} and \cite{CL07}. 

The anisotropic mean curvature flow (AMCF) is a family of immersions $X:M^n\times[0,T)\to\mathbb{R}^{n+1}$ satisfying the equation
	\begin{equation}
		\label{AMCF}
		\partial_t X = -H_F\nu_F,
	\end{equation}
	where $H_F$ is the anisotropic mean curvature of $M_t=X(M,t)$ and $\nu_F$ is the anisotropic normal vector.  We have the following evolution equations along AMCF. 
	
	\begin{lemma}
		Along AMCF \eqref{AMCF}, the following evolution equations hold:
		\begin{align}
            \partial_t \mathrm{d}\mu_F =& - H_F^2 \mathrm{d}\mu_F, &\label{equ-amcfarea}\\
			\partial_t h^j_i =& \nabla^j\nabla_i H_F + H_F (h^2)^j_i +  g^{jk}A_{pik}\nabla^p H_F, \label{equ-aniwei}\\
			\partial_t H_F =& \Delta H_F + H_F |h|_{g}^2 +  g^{ik}A_{pik}\nabla^p H_F. \label{equ-amceq}
		\end{align}
	\end{lemma}

	Let $\{e_{0},e_{1},...,e_{n}\}$ be a basis of $\mathbb{R}^{n+1}$ with dual basis $\{\phi^{0},\phi^{1},...,\phi^{n}\}$. Consider a graph-like hypersurface $M$ defined through local embeddings
	\begin{equation*}
		x=y^{i}e_{i}+u(y^{1},...,y^{n})e_{0},
	\end{equation*}
	where $y^1,\cdots,y^n$ are local coordinates on $M$. The unit normal field is given by
	\begin{equation*}
		\nu = \dfrac{Du - \phi^0}{\sqrt{1+|Du|^2}} = \dfrac{\sum_{i=1}^{n}u_{i}\phi^{i}-\phi^0}{\sqrt{1+|Du|^2}},
	\end{equation*}
	where $u_i=\dfrac{\partial u}{\partial y^i}$. The anisotropic mean curvature of $M$ and the squared norm of the anisotropic second fundamental form with respect to $g$ have the form 
	\begin{align*}
		H_F = &~u_{ij}D^iD^j F|_{Du-\phi^0},\\ 
        |h|_{g}^2 = & ~D^iD^kF|_{Du-\phi^0}D^jD^{\ell}F|_{Du-\phi^0} u_{ij}u_{k\ell},
	\end{align*}
    where $u_{ij}=\dfrac{\partial^2 u}{\partial y^i \partial y^j}$ denotes the second-order derivatives. Then the anisotropic mean curvature flow of the graph hypersurface $M$ is equivalent to the following PDE
	\begin{equation*}
		\partial_tu=FD^2F|_{Du-\phi^0}(\phi^i,\phi^j)D^2u(e_i,e_j).
	\end{equation*}

	For later use, set
	\begin{equation}
	\label{eq:graph-coefficients}
	 A^{ij}(p):=F(p-\phi^0)D^2F|_{p-\phi^0}(\phi^i,\phi^j).
	\end{equation}
	For every $K<\infty$ there are constants $0<\lambda_K\leq\Lambda_K<\infty$, depending only on $F$, $K$, and the fixed basis, such that
	\begin{equation}
	\label{eq:graph-uniform-parabolicity}
	 \lambda_K|\xi|^2\leq A^{ij}(p)\xi_i\xi_j
	 \leq\Lambda_K|\xi|^2,
	 \qquad |p|\leq K.
	\end{equation}
	Indeed, $p-\phi^0$ ranges over a compact subset of $\mathbb R^{n+1}\setminus\{0\}$, and no nonzero covector in $\operatorname{span}\{\phi^1,\ldots,\phi^n\}$ is parallel to $p-\phi^0$. Thus the strict convexity of $F$ in nonradial directions and compactness give \eqref{eq:graph-uniform-parabolicity}.
	
	\subsection{Weak inverse anisotropic mean curvature flow} \label{sec.2-wiamcf}

Suppose that $M_t$ is a smooth solution to the IAMCF \eqref{IAMCF}, and can be given by level sets $M_t=\partial E_t$ of a function $u:\mathbb{R}^{n+1} \to \mathbb{R}$, with $E_t=\{x\in\mathbb{R}^{n+1}:u(x)<t\}$. If $u$ is smooth and $Du\neq 0$, then \eqref{IAMCF} is equivalent to the degenerate elliptic equation 
\begin{equation}\label{WIAMCF}
    \text{div}(DF(Du))=F(Du)
\end{equation}
in $\mathbb{R}^{n+1}\setminus \Omega$ and $\Omega=\{u<0\}$.

The weak solution of \eqref{WIAMCF} was defined in \cite{DGX23} via the minimizing principle as in Huisken--Ilmanen \cite{HI01} for the isotropic case.  
Given an open bounded set $\Omega\subset \mathbb{R}^{n+1}$ with smooth boundary, a locally Lipschitz function $u$ on $\mathbb{R}^{n+1}$ is called a weak IAMCF starting from $\Omega$, if for every locally Lipschitz function $\varphi$  with $\{u\neq\varphi\}\subset\subset \mathbb{R}^{n+1}\setminus\Omega$ and any compact set $K\subset \mathbb{R}^{n+1}\setminus\Omega$ containing $\{u\neq\varphi\}$ we have 
  \begin{equation*}
        J_{F,u}^K(u) \leq J_{F,u}^K(\varphi)
    \end{equation*}
and $\Omega=\{u<0\}$, where the functional is defined as 
 \begin{equation*}
        J_{F,u}^K(\varphi) = \int_{K} \left[ F(D\varphi)+\varphi F(Du) \right] \mathrm{d}x.
\end{equation*}
 We say that $u$ is a proper solution if in addition $\lim\limits_{|x|\to+\infty}u(x)=+\infty$.

The weak solution of \eqref{WIAMCF} can also be defined by set functional.
\begin{definition}
\label{def-minimizes}
    We say that $E$ minimizes $J_{F,u}$ in a set $U$ (minimizes on the outside, minimizes on the inside, resp.) if
    \begin{equation*}
        J_{F,u}^K(E) \leq J_{F,u}^K(G)
    \end{equation*}
    for any $G$ such that $E\triangle G\subset\subset U$  ($G\supset E$, $G\subset E$ resp.) and any compact set $K$ containing $E\triangle G$. Here $E\triangle G = (E\backslash G)\cup(G\backslash E)$, and 
    \begin{equation*}
        J_{F,u}^K(G) = \int_{\partial^* G \cap K} F(\nu)\mathrm{d}\mathcal{H}^{n} - \int_{G\cap K} F(Du) \mathrm{d}x
    \end{equation*}
    for a set $G$ of locally finite perimeter, and $\partial^*G$ denotes the reduced boundary of $G$.
\end{definition}

Then a locally Lipschitz function $u$ on $\mathbb{R}^{n+1}$ is a weak solution of IAMCF if for each $t>0$, $E_t=\{u<t\}$ minimizes $J_{F,u}$ in $\mathbb{R}^{n+1}\setminus\Omega$ in the sense of Definition \ref{def-minimizes}. Therefore, we also say that $E_t$ is a weak solution of IAMCF with initial data $E_0=\Omega$. 

We have the following comparison principle:
\begin{lemma}[see {\cite[Proposition 3.3]{DGX23}}]
    \label{lem-weakcom}
    Let $\{E_t\}_{t>0}$ and $\{F_t\}_{t>0}$ be two weak solutions of IAMCF with initial data $E_0$, $F_0$ respectively, and $E_0\subset F_0$, then $E_t\subset F_t$ as long as $E_t$ is precompact.
    
    In particular, for a given $E_0$ there exists at most one solution $\{E_t\}_{t>0}$ of \eqref{WIAMCF} such that $E_t$ is precompact. 
\end{lemma}


The following lemma says that a classical solution to IAMCF \eqref{IAMCF} is also a weak solution to IAMCF \eqref{eq:WIAMCF}.
\begin{lemma}[see {\cite[Lemma 2.5]{GWZ26}}]
		\label{lem-smweak}
		Let $\{M_t\}_{c\leq t<d}$ be a smooth family of hypersurfaces with positive anisotropic mean curvature that solves \eqref{IAMCF} classically. Let $u=t$ on $M_t$, $u<c$ in the region bounded by $M_c$, and set $E_t=\{x\in\mathbb{R}^{n+1}:u(x)<t\}$. Then for $c\leq t<d$, $E_t$ minimizes $J_{F,u}$ in $E_d\backslash \overline{E_c}$ in the sense of Definition \ref{def-minimizes}. In other words, smooth flows of IAMCF satisfy the weak formulation in the domain they foliate.
	\end{lemma}

	The weak anisotropic mean curvature is defined by the first variational formula.
	\begin{definition}
		\label{def-wamc}
		Let $M^n\subset\mathbb{R}^{n+1}$ be a hypersurface of $C^1$ or $C^1$ with a small singular set and locally finite Hausdorff measure. A locally integrable function $H_F$ on $M$ is called weak anisotropic mean curvature if it satisfies
		\begin{equation}
			\int_{M^n} \mathrm{div}_{F,M}(V) F(\nu)\mathrm{d}\mathcal{H}^{n} = \int_{M^n} H_F \langle V,\nu \rangle \mathrm{d}\mathcal{H}^{n},
			\label{equ-wamc}
		\end{equation}
		for any vector fields $V\in C_c^{\infty}(\mathbb{R}^{n+1})$, where $\mathrm{div}_{F,M}(V)=\mathrm{div}(V)-\left\langle D_{\nu_F} V, \dfrac{\nu}{F(\nu)} \right\rangle$.
	\end{definition}
	
Let $u$ be a weak solution of \eqref{eq:WIAMCF} with initial data $E_0$ and let $M_t=\partial\{u<t\}$. Then for a.e. $t$, the weak anisotropic mean curvature $H_F$ of $M_t$ satisfies (see {\cite[Proposition 3.5]{DGX23}})
\begin{equation} \label{lem-level}
H_F = F(Du)\ \ a.e.\ x\in M_t.
\end{equation}

Finally, the following gradient estimate and asymptotic Wulff shape property of the weak IAMCF were proved recently by the authors \cite{GWZ26}, which will be used in Section \ref{sec5} to show that for sufficiently large time the level sets are strictly star-shaped. 

\begin{lemma}[see {\cite[Corollary 1.5]{GWZ26}}]
	\label{lem:weak-gradient}
	Let $u$ be the weak solution of IAMCF \eqref{eq:WIAMCF} starting from an open bounded set $\Omega\subset\mathbb{R}^{n+1}$ with smooth boundary. After translating the coordinates so that $0\in\Omega$, there exists a constant $C=C(n,F,\Omega)>0$ such that
	\begin{equation}
		\label{eq:weak-gradient}
		F(Du)(x)\leq \frac{C}{F^\circ(x)}
	\end{equation}
	for a.e. $x\in\mathbb{R}^{n+1}\setminus\Omega$.
\end{lemma}

    \begin{theorem}[{see \cite[Theorem 1.6]{GWZ26}}]
    \label{thm:anisotropic-asymptotic}
        Let $u$ be the weak solution of IAMCF \eqref{eq:WIAMCF} starting from an open bounded set $\Omega\subset\mathbb{R}^{n+1}$ with smooth boundary $\partial\Omega$. Then
	   \begin{equation*}
		  u(x) = n\log F^{\circ}(x)+\log\left( \dfrac{|\partial\mathcal{W}|_F}{|\partial\Omega^*|_F} \right) + o(1) \quad \text{as} \ F^\circ(x) \to \infty,
	   \end{equation*}
        where $\Omega^*$ is the strictly outward $F$-minimizing hull of $\Omega$. Moreover, the expanding Wulff shape is the only solution to \eqref{eq:WIAMCF} on $\mathbb{R}^{n+1}\backslash\{0\}$ with compact level sets.
    \end{theorem}

	\section{Huisken--Ilmanen type Harnack estimate}
	\label{sec3}
	
	In this section, we prove Theorem~\ref{prop-harnack}. The estimate shows that the upper bound for $1/H_F$ depends only on the initial star-shaped constants $R_1$, $R_2$, the dimension $n$, and the anisotropy $F$. In particular, the lower bound of $H_F$ along the flow \eqref{IAMCF} is independent of the initial curvature.

	The proof employs the maximum principle applied to a suitably chosen auxiliary function, with estimates performed at its space-time maximum point. This section proceeds in three parts. We first bound the anisotropic support function under strict star-shapedness. Next, we derive the evolution equation for $G(\nu_F)(X,X)$, the anisotropic analogue of $|X|^2$. Finally, we combine these results to prove Theorem~\ref{prop-harnack}.

	The following evolution equations along smooth IAMCF \eqref{IAMCF} are standard.
	\begin{lemma}[see {\cite[Proposition 4.1]{Xia17}}]
		Along IAMCF \eqref{IAMCF}, some quantities of smooth solutions $M_t$ with $H_F>0$ evolve by
		\begin{align}		
		\partial_t \nu_F = &\dfrac{\nabla H_F}{H_F^2}, \label{equ-nuf} \\
			\partial_t \mathrm{d}\mu_F =  &\mathrm{d}\mu_F, \label{s2.evl-dmu}\\ 
			\partial_t H_F = &\dfrac{1}{H_F^2}\left( \Delta H_F + g^{ik}A_{pik}\nabla^p H_F \right) - 2 \dfrac{|\nabla H_F|_{g}^2}{H_F^3} - \dfrac{|h|_{g}^2}{H_F},
			\label{equ-amc}\\
			\partial_t \chi = &\dfrac{1}{H_F^2}\left( \Delta\chi + g^{ik}A_{pik}\nabla^p \chi \right) + \dfrac{|h|_{g}^2}{H_F^2}\chi.
			\label{equ-asupp} 
        \end{align}	
\end{lemma}
	
	\subsection{Bounds on the anisotropic support function}

	\begin{lemma}
		\label{lem-anisosupp}
		If the initial hypersurface $M_0$ satisfies
		\begin{equation*}
			0 < R_1 \leq \chi \leq R_2,
		\end{equation*}
		then the solution $M_t$ of IAMCF \eqref{IAMCF} satisfies
		\begin{equation}
			\label{equ-asest}
			R_1 \mathrm{e}^{\frac{t}{n}} \leq \chi \leq F^{\circ}(X) \leq R_2 \mathrm{e}^{\frac{t}{n}}
		\end{equation}
		on $M_t$ for every time for which the smooth solution exists.
	\end{lemma}
	
	\begin{proof}
		We first prove the lower bound of $\chi$. By the inequality $|h|_{g}^2\geq {H_F^2}/{n}$,  the evolution equation \eqref{equ-asupp} of the anisotropic support function $\chi$ gives
		\begin{equation*}
			\partial_t \chi \geq \dfrac{1}{H_F^2}\left( \Delta \chi + g^{ik}A_{pik}\nabla^p \chi \right) + \dfrac{1}{n} \chi.
		\end{equation*}
		Since initially $\chi\geq R_1$, the maximum principle implies the lower bound
		\begin{equation*}
			\chi\geq R_1\mathrm{e}^{\frac{t}{n}}\ \ \mbox{on}\  M_t.
		\end{equation*}
		
		Next, we derive an upper bound for the anisotropic distance $F^{\circ}(X)$.  Along the flow \eqref{IAMCF}, the evolution equation of $F^{\circ}(X)$ is given by
		\begin{equation}
			\partial_t \left(  F^{\circ}(X) \right) = DF^{\circ}(X)\cdot \partial_t X = \dfrac{DF^{\circ}(X)\cdot \nu_F}{H_F}.
			\label{equ-F0evo}
		\end{equation}
		Let $\alpha(t)=\max\limits_{M_t}F^{\circ}(X)$ be the maximum value of the anisotropic distance on $M_t$ at time $t$, which is positive, and let $P$ be a maximizing point. Then $M_t$ lies inside the scaled Wulff shape $\alpha(t)\mathcal{W}$ and is tangent to it at $P$ from the inside. The comparison implies that at the point $P$, 
		\begin{equation}
			H_F \geq n/\alpha(t).
			\label{equ-atP1}
		\end{equation}
		Moreover, the anisotropic normal of $M_t$ at $P$ coincides with the radial direction
		\begin{equation*}
			\nu_F(P)=\dfrac{P}{F^{\circ}(P)},
		\end{equation*}
		which follows from the 1-homogeneity of $F^{\circ}$. Using property \eqref{fprop1} we compute
		\begin{equation}
			DF^{\circ}(P)\cdot \nu_F(P) = DF^{\circ}(P)\cdot \dfrac{P}{F^{\circ}(P)}=1.
			\label{equ-atP2}
		\end{equation}
		Combining \eqref{equ-F0evo}, \eqref{equ-atP1} and \eqref{equ-atP2}, Hamilton's maximum principle \cite[Lemma 3.5]{Ha86} gives,
		\begin{equation}\label{s3.dalphat}
			\frac{d}{dt}\alpha(t)\leq\frac{1}{n}\alpha(t).
		\end{equation}
		At a point where $F^\circ(X)$ attains its initial maximum, $M_0$ is tangent from the inside to a Wulff shape centered at the origin. At this point $\nu_F=X/F^\circ(X)$, and hence $\chi=F^\circ(X)$. Therefore $\max_{M_0}F^\circ(X)=\max_{M_0}\chi\leq R_2$. The inequality \eqref{s3.dalphat} implies
		\begin{equation*}
			F^{\circ}(X)\leq R_2\mathrm{e}^{\frac{t}{n}}\ \ \ \ \mbox{on}\ M_t.
		\end{equation*}		
		Finally, by the inequality $\chi \leq F^{\circ}(X)$ (see \eqref{equ-uf0}), we complete the proof.
	\end{proof}
	
	\subsection{\texorpdfstring{Evolution equation of $G(\nu_F)(X,X)$}{Evolution equation of G(nuF)(X,X)}}
	
	\begin{lemma}
		Along the flow \eqref{IAMCF}, the quantity $G(\nu_F)(X,X)$ evolves by
		\begin{align}
			&\left( \partial_t - \dfrac{1}{H_F^2}\Delta - \dfrac{1}{H_F^2} g^{ik}A_{pik}\nabla^p  \right)G(\nu_F)(X,X)\notag\\
			=& \dfrac{4\chi}{H_F}-\dfrac{2n}{H_F^2} -\dfrac{2}{H_F^2}h^{p\ell}Q(\nu_F)(X_p,X_{\ell},X) \notag\\
			&+\dfrac{1}{H_F^2}h^s_ph^{pk}g^{q\ell}Q_{kqs}Q(\nu_F)(X_{\ell},X,X) \nonumber\\
			&- \dfrac{1}{H_F^2} h^{p\ell}h^k_p T(\nu_F)(X_k,X_{\ell},X,X), \label{equ-evolutiong}
		\end{align}
		where we denote $h^{ij}=g^{ik}h_{k\ell}g^{\ell j}$.
	\end{lemma}
	\begin{proof}
		We compute at a fixed point and choose an orthonormal basis with respect to $\nabla$ around that point. We will start with the time derivative, then work out the first and second covariant derivatives, take the Laplacian, add the first-order term, handle the derivative of $Q(\nu_F)(X_{\ell},X,X)$, and finally put everything together to obtain the evolution equation \eqref{equ-evolutiong}.
		
		First, using the flow equation \eqref{IAMCF} and evolution equation \eqref{equ-nuf} for $\nu_F$, we obtain
		\begin{align}
			\partial_t \left( G(\nu_F)(X,X) \right) =& 2 G(\nu_F)(\partial_t X,X) + 	Q(\nu_F)(\partial_t\nu_F,X,X) \notag\\
			=& 2 \dfrac{\chi}{H_F} + \dfrac{1}{H_F^2}\nabla^{\ell} H_F Q(\nu_F)(X_\ell,X,X).  \label{equ-ptg}
		\end{align}
		Next, the Weingarten formula \eqref{weinformula} gives the first covariant derivative:
		\begin{align}
			\nabla_p\left( G(\nu_F)(X,X) \right) =& 2G(\nu_F)(X_p,X) + Q(\nu_F)(\partial_p \nu_F,X,X) \notag\\
			=& 2G(\nu_F)(X_p,X) + h^{\ell}_p Q(\nu_F)(X_{\ell},X,X). \label{equ-dpg}
		\end{align}		
		For the second covariant derivative, we employ the Gauss-Weingarten formula \eqref{gaussformula}, \eqref{weinformula} and the Codazzi equation \eqref{codazziequ}. A direct computation gives
		\begin{align}
			&\nabla_q\nabla_p\left( G(\nu_F)(X,X) \right) \notag\\
			=& 2G(\nu_F)(\partial_q\partial_p X,X) + 2g_{pq} + 2h_q^{\ell}Q(\nu_F)(X_{\ell},X_p, X) \notag\\
			&+ \nabla_qh^{\ell}_p Q(\nu_F)(X_{\ell},X,X) + h^{\ell}_p \nabla_q \left(Q(\nu_F)(X_{\ell},X,X)\right) \notag\\
			=& -2h_{pq}\chi + 2g^{k\ell}A_{pqk}G(\nu_F)(X_{\ell},X)\nonumber\\
			&+2g_{pq} + 2h_q^{\ell}Q(\nu_F)(X_{\ell},X_p, X)\notag\\
			&+\left( \nabla^{\ell}h_{pq} + g^{m\ell}h^s_q A_{smp} - h^{\ell s}A_{sqp} \right)Q(\nu_F)(X_{\ell},X,X) \notag\\
			& + h^{\ell}_p\nabla_q\left( Q(\nu_F)(X_{\ell},X,X) \right). \label{equ-hessg}
		\end{align}
		
		Taking the trace of \eqref{equ-hessg} with respect to the metric $g$ yields
		\begin{align}
			\Delta\left( G(\nu_F)(X,X) \right) 
			=& g^{pq}\nabla_q\nabla_p\left( G(\nu_F)(X,X) \right) \notag\\
			=& -2 H_F \chi + 2n + 2g^{pq}g^{k\ell} A_{pqk}G(\nu_F)(X_{\ell},X) \nonumber\\
			&+ 2h^{p\ell}Q(\nu_F)(X_{\ell},X_p,X)\notag\\
			&+\left( \nabla^{\ell}H_F + g^{m\ell}h^{sp} A_{smp} - g^{pq}h^{\ell s}A_{sqp} \right)Q(\nu_F)(X_{\ell},X,X) \notag\\
			&+h^{p\ell}\nabla_p\left( Q(\nu_F)(X_{\ell},X,X) \right). \label{equ-laplaciang}
		\end{align}
		Combining the two equations \eqref{equ-dpg} and \eqref{equ-laplaciang}, we find
		\begin{align}
			&\Delta\left( G(\nu_F)(X,X) \right) + g^{ik}A_{pik}\nabla^p\left( G(\nu_F)(X,X) \right) \notag\\
			=& -2H_F \chi + 2n + \nabla^{\ell}H_F Q(\nu_F)(X_{\ell},X,X) \nonumber\\
			& \underbrace{+ 2 h^{p\ell}Q(\nu_F)(X_{\ell},X_p,X) +\left( 2g^{pq}g^{k\ell}A_{pqk} + 2g^{ik}g^{p\ell}A_{pik} \right)G(\nu_F)(X_{\ell},X) }_{(a)} \nonumber\\
			& + g^{m\ell}h^{ps}A_{smp} Q(\nu_F)(X_{\ell},X,X) + h^{p\ell}\nabla_p\left( Q(\nu_F)(X_{\ell},X,X) \right). \label{s3.Delta-G}
		\end{align}
		The term $(a)$ vanishes. Indeed, using \eqref{equ-q}, the decomposition of $X$
		\begin{equation*}
			X=g^{k\ell}G(\nu_F)(X,X_k)X_{\ell}+G(\nu_F)(X,\nu_F)\nu_F,
		\end{equation*}
		the symmetry of $Q$ and the definition of $A_{ijk}$ (see \eqref{Atensor}), one finds
		\begin{equation}\label{s3.a-vanish}
			(a) = 2\left( A_{pik} + A_{ikp} + h_{i}^mQ_{kpm} \right) g^{p\ell}g^{ik}G(\nu_F)(X_{\ell},X) =0.
		\end{equation}
	To handle the $\nabla Q$ term in \eqref{s3.Delta-G}, we apply property \eqref{equ-q} and the Gauss-Weingarten formula \eqref{gaussformula}  \eqref{weinformula} to obtain
		\begin{align}
			h^{p\ell}\nabla_p\left( Q(\nu_F)(X_{\ell},X,X) \right) =& h^{p\ell}Q(\nu_F)(\partial_p\partial_{\ell} X,X,X) \notag\\
			&+ 2h^{p\ell}Q(\nu_F)(X_{\ell},X_p,X) \notag\\
			&+ h^{p\ell}T(\nu_F)(\partial_p \nu_F,X_{\ell},X,X) \notag\\
			=& g^{q\ell}h^{pk}A_{pkq}Q(\nu_F)(X_{\ell},X,X) \nonumber\\
			&+ 2h^{p\ell}Q(\nu_F)(X_{\ell},X_{p},X) \notag\\
			& + h^{p\ell}h^k_p T(\nu_F)(X_k,X_{\ell},X,X). \label{equ-dq}
		\end{align}
        Substituting \eqref{s3.a-vanish} and \eqref{equ-dq} into \eqref{s3.Delta-G} gives
		\begin{align*}
			& \Delta\left( G(\nu_F)(X,X) \right) + g^{ik}A_{pik}\nabla^p\left( G(\nu_F)(X,X) \right) \notag\\
			=& -2H_F \chi + 2n + \nabla^{\ell}H_F Q(\nu_F)(X_{\ell},X,X) \nonumber\\
			& + \underbrace{\left(g^{m\ell}h^{ps}A_{smp}+g^{q\ell}h^{pk}A_{pkq}\right)Q(\nu_F)(X_{\ell},X,X)}_{(b)} \nonumber\\
			& + 2h^{p\ell}Q(\nu_F)(X_{\ell},X_{p},X) + h^{p\ell}h^k_p T(\nu_F)(X_k,X_{\ell},X,X).
		\end{align*}
		Using the explicit expression of $A_{ijk}$ in \eqref{Atensor}, the term $(b)$ simplifies to
		\begin{equation*}
			(b)	= -h^s_p h^{pk}g^{q\ell} Q_{kqs}Q(\nu_F)(X_{\ell},X,X).
		\end{equation*}
		Therefore,
		\begin{align}
			& \Delta\left( G(\nu_F)(X,X) \right) + g^{ik}A_{pik}\nabla^p\left( G(\nu_F)(X,X) \right) \notag\\
			=& -2H_F \chi + 2n + \nabla^{\ell}H_F Q(\nu_F)(X_{\ell},X,X) \nonumber\\
			& -h^s_p h^{pk}g^{q\ell} Q_{kqs}Q(\nu_F)(X_{\ell},X,X) \nonumber\\
			& + 2h^{p\ell}Q(\nu_F)(X_{\ell},X_{p},X) + h^{p\ell}h^k_p T(\nu_F)(X_k,X_{\ell},X,X). \label{equ-hess}
		\end{align}		
		Finally, combining \eqref{equ-ptg} and \eqref{equ-hess}, we obtain \eqref{equ-evolutiong}. 
	\end{proof}
	
	\subsection{Proof of Theorem \ref{prop-harnack}}
	
		Since \eqref{equ-asest} holds for all times under consideration, we set $w:=\chi^{-1}$ and consider the auxiliary function
		\begin{equation}
			\label{equ-W}
			W:= \dfrac{\varphi^{1-\varepsilon}(w)\mathrm{e}^{\gamma G(\nu_F)(X,X)}}{H_F},
		\end{equation}
		where $\varphi:=\varphi(w)$ is a function to be chosen, and $\gamma>0$, $\varepsilon\in(0,1)$ are constants to be determined later.
		
		The main idea is to apply the parabolic maximum principle at a maximum point of $\ln(tW)$ on $M\times[0,T]$. This converts the initial star-shaped bounds $R_1 \leq \chi \leq R_2$ into an explicit upper bound for $1/H_F$. The exponential growth of $\chi$ (see \eqref{equ-asest}) then propagates the estimate to all times. We divide the proof into several steps.
		
		\textbf{Step 1. Evolution equation for $\ln W$.} Denote the operator
		\begin{equation*}
			\mathscr{L}:=\partial_t  - \dfrac{1}{H_F^2}\Delta- \dfrac{1}{H_F^2}g^{ik}A_{pik}\nabla^p.
		\end{equation*}
		By \eqref{equ-asupp},
		\begin{equation*}
			\mathscr{L} w = - \dfrac{|h|_{g}^2}{H_F^2}w - \dfrac{2|\nabla w|_{g}^2}{w H_F^2}.
		\end{equation*}
		Hence on $\{\varphi\neq 0\}$ we have
		\begin{align}
			\mathscr{L} \ln\varphi =& -\dfrac{|h|_{g}^2}{H_F^2}\dfrac{\varphi' w}{\varphi} - \dfrac{|\nabla w|_{g}^2}{H_F^2}\left( 2\dfrac{\varphi'}{w\varphi} + \dfrac{\varphi''}{\varphi} - \dfrac{\varphi'^2}{\varphi^2} \right). \label{equ-lnphi0}
		\end{align}
		To simplify the gradient terms in \eqref{equ-lnphi0}, we choose
		\begin{equation*}
			\varphi(s):= \dfrac{s}{2R_1^{-1}-s}.
		\end{equation*}
		Then $\varphi:=\varphi(w)$ is well defined since
		\begin{equation}
			R_2^{-1}\mathrm{e}^{-\frac{t}{n}} \leq w \leq R_1^{-1}\mathrm{e}^{-\frac{t}{n}} \leq R_1^{-1}.
			\label{equ-west}
		\end{equation}
		Under the notation $\varphi'=\varphi'(w)$ and $\varphi''=\varphi''(w)$, and after a direct computation, we have
		\begin{equation}
			\dfrac{\varphi' w}{\varphi} = \dfrac{2}{2-w R_1}\ \ \text{and}\ \ 2\dfrac{\varphi'}{w\varphi} + \dfrac{\varphi''}{\varphi} - \dfrac{\varphi'^2}{\varphi^2} = \dfrac{\varphi'^2}{\varphi^2}.
			\label{equ-phideri}
		\end{equation}
		Combining \eqref{equ-lnphi0} and \eqref{equ-phideri} yields
		\begin{equation}
			\mathscr{L}\ln\varphi = -\dfrac{|h|_{g}^2}{H_F^2}\dfrac{\varphi' w}{\varphi} -\dfrac{1}{H_F^2}\dfrac{|\nabla\varphi|^2_{g}}{\varphi^2}. \label{equ-lnphi}
		\end{equation}
		Moreover, evolution equation \eqref{equ-amc} implies
		\begin{equation}
			\mathscr{L} \ln H_F = -\dfrac{|\nabla H_F|^2_{g}}{H_F^4} - \dfrac{|h|_{g}^2}{H_F^2}.
			\label{equ-lnhf}
		\end{equation}
		Therefore, combining \eqref{equ-evolutiong}, \eqref{equ-lnphi} and \eqref{equ-lnhf}, we obtain
		\begin{align}
			\mathscr{L} \ln W =& \left[ \dfrac{|\nabla H_F|^2_{g}}{H_F^4}+\dfrac{|h|_{g}^2}{H_F^2} \right] - (1-\varepsilon) \left[ \dfrac{|h|_{g}^2}{H_F^2}\dfrac{\varphi' w}{\varphi} + \dfrac{1}{H_F^2}\dfrac{|\nabla\varphi|^2_{g}}{\varphi^2}  \right] \notag\\
			&+\gamma\left( \dfrac{4}{H_F w}-\dfrac{2n}{H_F^2} \right) - \dfrac{2\gamma}{H_F^2}h^{p\ell}Q(\nu_F)(X_p,X_{\ell},X)\notag\\
			&+ \gamma\dfrac{1}{H_F^2}h^s_ph^{pk}g^{q\ell}Q_{kqs}Q(\nu_F)(X_{\ell},X,X)\nonumber\\
			&- \gamma\dfrac{1}{H_F^2} h^{p\ell}h^k_p T(\nu_F)(X_k,X_{\ell},X,X).  \label{equ-lnw}
		\end{align}
		
		\textbf{Step 2. Gradient estimate at a critical point.}
		Suppose the positive maximum of $\tilde{W}:=tW$ on $M\times[0,T]$ is attained at some point $(p_0,t_0)$ with $t_0>0$. If $t_0=T$, time derivatives below are understood from the left. The same parabolic maximum-principle inequality holds. At this maximum point we have
		\begin{equation*}
			\nabla \ln W(p_0,t_0)=0.
		\end{equation*}
		This implies
		\begin{equation*}
			0 = \dfrac{\nabla W}{W} = (1-\varepsilon)\dfrac{\nabla\varphi}{\varphi} + \gamma \dfrac{\nabla\left( \mathrm{e}^{G(\nu_F)(X,X)} \right)}{\mathrm{e}^{G(\nu_F)(X,X)}} + \dfrac{\nabla H_F^{-1}}{H_F^{-1}}.
		\end{equation*}
		By the Cauchy--Schwarz inequality, we obtain
		\begin{align}
			\left|\dfrac{\nabla H_F^{-1}}{H_F^{-1}}\right|^2_{g} =& \left| (1-\varepsilon)\dfrac{\nabla\varphi}{\varphi} + \gamma\dfrac{\nabla\left( \mathrm{e}^{G(\nu_F)(X,X)} \right)}{\mathrm{e}^{G(\nu_F)(X,X)}} \right|_{g}^2 \notag\\
			=& (1-\varepsilon)^2\left| \dfrac{\nabla\varphi}{\varphi} \right|_{g}^2 + \gamma^2\left| \dfrac{\nabla\left( \mathrm{e}^{G(\nu_F)(X,X)} \right)}{\mathrm{e}^{G(\nu_F)(X,X)}} \right|^2_{g} \nonumber\\
			&+ 2(1-\varepsilon)\gamma\left\langle \dfrac{\nabla\varphi}{\varphi},\dfrac{\nabla\left( \mathrm{e}^{G(\nu_F)(X,X)} \right)}{\mathrm{e}^{G(\nu_F)(X,X)}}  \right\rangle_{g} \notag\\
			\leq& \bigg( (1-\varepsilon)^2+\varepsilon(1-\varepsilon) \bigg)\left| \dfrac{\nabla\varphi}{\varphi} \right|_{g}^2 \notag\\
			&+ \left(1+\dfrac{1-\varepsilon}{\varepsilon}\right)\gamma^2\left| \dfrac{\nabla\mathrm{e}^{G(\nu_F)(X,X)}}{\mathrm{e}^{G(\nu_F)(X,X)}} \right|_{g}^2 \notag\\
			=& (1-\varepsilon)  \dfrac{|\nabla\varphi|_{g}^2}{\varphi^2} + \varepsilon^{-1}\gamma^2 \dfrac{|\nabla\mathrm{e}^{G(\nu_F)(X,X)}|_{g}^2}{\mathrm{e}^{2G(\nu_F)(X,X)}}.  \label{equ-nablahf-1}
		\end{align}
		Therefore, at $(p_0,t_0)$, substituting \eqref{equ-nablahf-1} into \eqref{equ-lnw} and using \eqref{equ-phideri} gives
		\begin{align}
			\mathscr{L} \ln W \leq& -\left( \dfrac{wR_1-2\varepsilon}{2-wR_1} \right)\dfrac{|h|_{g}^2}{H_F^2} +\gamma\left( \dfrac{4}{H_F w}-\dfrac{2n}{H_F^2} \right) \nonumber\\
			&+ \varepsilon^{-1}\gamma^2\dfrac{|\nabla G(\nu_F)(X,X)|_{g}^2}{H_F^2} \notag\\
			&- \gamma\dfrac{2}{H_F^2}\underbrace{h^{p\ell}Q(\nu_F)(X_p,X_{\ell},X)}_{(I)} \nonumber\\
			&+ \gamma\dfrac{1}{H_F^2}\underbrace{h^s_ph^{pk}g^{q\ell}Q_{kqs}Q(\nu_F)(X_{\ell},X,X)}_{(II)} \notag\\
			&- \gamma\dfrac{1}{H_F^2} \underbrace{h^{p\ell}h^k_p T(\nu_F)(X_k,X_{\ell},X,X)}_{(III)}.
			\label{12}
		\end{align}
		
		\textbf{Step 3. Estimates for the anisotropic terms.}
		We estimate the anisotropic terms. By \eqref{equ-dpg},
		\begin{align}
			|\nabla G(\nu_F)(X,X)|_{g}^2
			=& g^{pq}\nabla_p\left( G(\nu_F)(X,X) \right) \nabla_q\left( G(\nu_F)(X,X) \right) \notag\\
			=&  4\underbrace{g^{pq}G(\nu_F)(X_p,X)G(\nu_F)(X_q,X)}_{(IV)} \nonumber\\
			&+ \underbrace{4h^{p\ell} G(\nu_F)(X_p,X) Q(\nu_F)(X_{\ell},X,X)}_{(V)} \notag\\
			& + \underbrace{h^{q\ell}h^s_q Q(\nu_F)(X_{\ell},X,X)Q(\nu_F)(X_s,X,X)}_{(VI)}.  \label{345}
		\end{align}
		
		Since $F$ is a Minkowski norm and $F^\circ(\nu_F)=1$, the vector $\nu_F$ lies on the compact Wulff shape $\partial\mathcal{W}$. Hence, all derivatives of $F^\circ$ and hence the associated tensors $G$, $Q$, $T$ are uniformly bounded on $\partial\mathcal{W}$, yielding a constant $C(n,F)$ that depends only on $n$ and $F$. In the following estimates, $C(n,F)$ denotes such a constant, which may change from line to line but always depends only on $n$ and $F$. We also use the equivalence of norms in finite dimension:
		\begin{equation*}
			|Y|\leq C(F)F^\circ(Y)\qquad\text{for all }Y\in\mathbb R^{n+1}.
		\end{equation*}
		Together with \eqref{equ-asest}, this gives $|X|\leq C(F)R_2\mathrm{e}^{T/n}$ on $M\times[0,T]$. In a $g$-orthonormal frame the tangent vectors $X_i$ have Euclidean length bounded by a constant depending only on $F$, because $G(\nu_F)$ is uniformly equivalent to the Euclidean metric on $\partial\mathcal W$. Using these facts, we first bound $(IV)$:
		\begin{align}
			\left|(IV)\right|=&g^{pq}G(\nu_F)(X_p,X)G(\nu_F)(X_q,X) \nonumber\\
			\leq& G(\nu_F)(X,X) \notag\\
			\leq& C(n,F)|X|^2 \nonumber\\
			\leq& C(n,F)R_2^2 \mathrm{e}^{2\frac{T}{n}}. \label{3est}
		\end{align}
		For the remaining terms, we apply \eqref{equ-asest} together with the Cauchy--Schwarz inequality:
		\begin{align}	\label{1245est}
				\left|(I)\right| \leq&~  C(n,F)|h|_{g}|X|
				\leq ~C(n,F)\varepsilon\gamma^{-1}|h|_{g}^2+ C(n,F)\varepsilon^{-1}\gamma R_2^2\mathrm{e}^{2\frac{T}{n}},\nonumber\\
			\left|	(II)\right|, \left|(III)\right| \leq &~C(n,F) |h|^2_{g} |X|^2 \leq ~C(n,F) |h|^2_{g} R_2^2\mathrm{e}^{2\frac{T}{n}},\nonumber\\
				\left|(V)\right| \leq &~C(n,F) |h|_{g} |X|^3 
				\leq ~C(n,F) |h|^2_{g} R_2^4\mathrm{e}^{4\frac{T}{n}} + C(n,F)R_2^2\mathrm{e}^{2\frac{T}{n}},\nonumber\\
				\left|(VI)\right| \leq &~C(n,F) |h|^2_{g} |X|^4
				\leq ~C(n,F) |h|^2_{g} R_2^4\mathrm{e}^{4\frac{T}{n}}.
		\end{align}
		Here the estimates for $(I)$ involve the parameters $\varepsilon,\gamma$ to be chosen later.
		
		Now we combine the evolution equation \eqref{12} with the bounds above \eqref{345}, \eqref{3est} and \eqref{1245est}. At $(p_0,t_0)$ we obtain
		\begin{align}
			\mathscr{L} \ln W \leq& \underbrace{ \left(-\dfrac{wR_1-2\varepsilon}{2-wR_1} + C_1\varepsilon + C_2\gamma R_2^2 \mathrm{e}^{2\frac{T}{n}} + \varepsilon^{-1}\gamma^2 C_3R_2^4\mathrm{e}^{4\frac{T}{n}} \right)}_{(A)}\dfrac{|h|_{g}^2}{H_F^2} \notag\\
			&+\gamma\underbrace{\left( \dfrac{4}{H_F w}-\dfrac{2n}{H_F^2} + \dfrac{\varepsilon^{-1}\gamma C_4R_2^2\mathrm{e}^{2\frac{T}{n}}}{H_F^2} \right)}_{(B)}, \label{equ-lnww}
		\end{align}
		where $C_1,C_2,C_3,C_4$ are constants depending only on $n$ and $F$.
		
		\textbf{Step 4. Parameter choice.}
		We first require that $\varepsilon$ and $\gamma$ satisfy
		\begin{equation}
			\varepsilon^{-1}\gamma C_4(n,F)R_2^2\mathrm{e}^{2\frac{T}{n}}=n.
			\label{equ-epsilongamma}
		\end{equation}
		Then the term $(B)$ becomes
		\begin{align*}
			(B) =& - \dfrac{n}{H_F^2} + \dfrac{4}{H_F w} 
			\leq  -\dfrac{n}{2H_F^2} + \dfrac{8}{n}R_2^2\mathrm{e}^{2\frac{T}{n}}, 
		\end{align*}
		where we used the Cauchy--Schwarz inequality and \eqref{equ-asest} to obtain
		\begin{equation*}
			\dfrac{4}{H_F w} \leq \dfrac{8}{nw^2} + \dfrac{n}{2H_F^2} \leq \dfrac{8}{n}R_2^2\mathrm{e}^{2\frac{T}{n}} + \dfrac{n}{2H_F^2}.
		\end{equation*}
		
		We then choose appropriate $\varepsilon$ to ensure that $(A)$ is nonpositive. Note that \eqref{equ-west} implies
		\begin{equation*}
			\dfrac{R_1}{R_2\mathrm{e}^{\frac{T}{n}}}\leq wR_1 \leq 1,
		\end{equation*}
		then
		\begin{equation}
			\dfrac{wR_1-2\varepsilon}{2-wR_1}\geq \dfrac{R_1-2\varepsilon R_2\mathrm{e}^{\frac{T}{n}}}{2R_2\mathrm{e}^{\frac{T}{n}} - R_1}.
			\label{equ-wR1lbd}
		\end{equation}
		Hence \eqref{equ-epsilongamma} and \eqref{equ-wR1lbd} yields
		\begin{equation*}
			(A) \leq -\dfrac{R_1}{2R_2\mathrm{e}^{\frac{T}{n}} - R_1} + \varepsilon \left( \dfrac{2R_2\mathrm{e}^{\frac{T}{n}}}{2R_2\mathrm{e}^{\frac{T}{n}}-R_1}+C_5(n,F) \right).
		\end{equation*}
		We can choose
		\begin{equation}
			\varepsilon=\dfrac{1}{2}\dfrac{R_1}{2(1+C_5(n,F))R_2\mathrm{e}^{\frac{T}{n}}-C_5(n,F)R_1}\in (0,1)  \label{equ-epsilon}
		\end{equation}
		such that $(A)\leq 0$.
	
        Therefore, at $(p_0,t_0)$,
		\begin{equation}
			\mathscr{L}\ln W \leq \gamma\left( -\dfrac{n}{2H_F^2} + \dfrac{8}{n}R_2^2\mathrm{e}^{2\frac{T}{n}} \right). \label{equ-lnwcrit}
		\end{equation}
		
		\textbf{Step 5. Maximum principle argument.}
		We now apply the maximum principle to the function $\tilde{W}=tW$ on $M\times[0,T]$. Suppose that the positive maximum of $\tilde{W}$ is attained at some point $(p_0,t_0)$ with $t_0>0$. Using \eqref{equ-lnwcrit} at $(p_0,t_0)$, we have
		\begin{equation}
			0 \leq \mathscr{L}\ln \tilde{W}(p_0,t_0) \leq  \gamma\left( -\dfrac{n}{2H_F^2} + \dfrac{8}{n}R_2^2\mathrm{e}^{2\frac{T}{n}} \right) + \dfrac{1}{t_0}.
			\label{tildew}
		\end{equation}
		Since $\varphi$ is increasing in $w$, inequality \eqref{equ-west} implies
		\begin{equation*}
			\dfrac{R_1}{2R_2\mathrm{e}^{\frac{T}{n}}} \leq \varphi\left( \left(R_2\mathrm{e}^{\frac{T}{n}}\right)^{-1} \right) \leq\varphi(w) \leq \varphi(R_1^{-1}) = 1.
		\end{equation*}
		Furthermore, the choice of $\gamma$ in \eqref{equ-epsilongamma}, the bound \eqref{equ-asest}, and the equivalence of $G(\nu_F)$ with the Euclidean metric give
		\begin{equation*}
			0\leq \gamma G(\nu_F)(X,X)
			\leq C(n,F)\gamma R_2^2\mathrm{e}^{2T/n}
			\leq C(n,F)\varepsilon
			\leq C(n,F).
		\end{equation*}
		Hence
		\begin{equation}
			1\leq \mathrm{e}^{2\gamma G(\nu_F)(X,X)}\leq C(n,F).
			\label{equ-GGGbd}
		\end{equation}
		Combining the definition \eqref{equ-W}, \eqref{tildew}--\eqref{equ-GGGbd}, and the choice of $\varepsilon$ in \eqref{equ-epsilon}, we deduce
		\begin{align*}
			\tilde{W}^2(p_0,t_0)\leq & C(n,F)\left(R_2\mathrm{e}^{\frac{T}{n}}\right)^2\left( T^2 + \varepsilon^{-1}T \right) \nonumber\\
			\leq & C(n,F)\left(R_2\mathrm{e}^{\frac{T}{n}}\right)^2\left( T^2 + \dfrac{R_2}{R_1}\mathrm{e}^{\frac{T}{n}}T \right)\nonumber\\
			\leq & C(n,F) \left(R_2 \mathrm{e}^{\frac{T}{n}}\right)^2\dfrac{R_2}{R_1}\mathrm{e}^{\frac{T}{n}}T^2\left( 1 + \dfrac{1}{T} \right). \nonumber
		\end{align*}
		Consequently, for any $p\in M_T$,
		\begin{align*}
			\dfrac{1}{H_F^2}(p,T)T^2\left(\dfrac{R_1}{2R_2\mathrm{e}^{\frac{T}{n}}} \right)^{2-2\varepsilon} \leq& \tilde{W}^2(p,T) \\
			\leq & C(n,F) \left(R_2 \mathrm{e}^{\frac{T}{n}}\right)^2 \dfrac{R_2}{R_1}\mathrm{e}^{\frac{T}{n}}T^2\left( 1 + \dfrac{1}{T} \right).
		\end{align*}
		Therefore, for every $T>0$ and every $p\in M_T$, we have the estimate
		\begin{equation}
			\dfrac{1}{H_F^2}(p,T) \leq C(n,F)\left( \dfrac{R_2}{R_1}\mathrm{e}^{\frac{T}{n}} \right)^{3-2\varepsilon}\left( R_2 \mathrm{e}^{\frac{T}{n}}\right)^2 \left(1+\dfrac{1}{T}\right).
			\label{equ-hf}
		\end{equation}
		
		\textbf{Step 6. Complete the proof.}
		The derivation of \eqref{equ-hf} uses only the support bounds at the initial time $t=0$, namely, $R_1\leq\chi\leq R_2$. However, \eqref{equ-asest} tells us that for any later time $s\geq 0$, the anisotropic support function satisfies
		\begin{equation*}
			R_1\mathrm{e}^{\frac{s}{n}}\leq \chi\leq R_2\mathrm{e}^{\frac{s}{n}}.
		\end{equation*}
		Thus we may restart the flow at time $s$ using these new bounds and apply the same argument on the time interval $[s,s+\tau]$. This yields an estimate that depends on the bounds at time $s$ rather than at time $0$: for any $s\geq 0$ and any $\tau>0$, we have
		\begin{equation*}
			\frac{1}{H_F^2}(p,s+\tau)\leq C(n,F)\left(\frac{R_2\mathrm{e}^{\frac{s}{n}}}{R_1\mathrm{e}^{\frac{s}{n}}}\mathrm{e}^{\frac{\tau}{n}}\right)^{3-2\varepsilon_{\tau}}
			\left(R_2\mathrm{e}^{\frac{s+\tau}{n}}\right)^2\left(1+\frac{1}{\tau}\right),
		\end{equation*}
		where the parameter $\varepsilon_{\tau}$ is now given by
		\begin{align*}
			\varepsilon_{\tau}=& \frac{1}{2}\,\frac{R_1\mathrm{e}^{\frac{s}{n}}}{2(1+C_5)R_2\mathrm{e}^{\frac{s}{n}}\mathrm{e}^{\frac{\tau}{n}}-C_5R_1\mathrm{e}^{\frac{s}{n}}}\\
			=& \frac{1}{2}\,\frac{R_1}{2(1+C_5)R_2\mathrm{e}^{\frac{\tau}{n}}-C_5R_1}\in (0,1),
		\end{align*}
		which is independent of $s$. Hence we obtain the estimate
		\begin{equation}
			\dfrac{1}{H_F^2}(p,s+\tau) \leq C(n,F)\left( \dfrac{R_2}{R_1}\mathrm{e}^{\frac{\tau}{n}} \right)^{3-2\varepsilon_{\tau}} \left( R_2 \mathrm{e}^{\frac{s+\tau}{n}} \right)^2 \left( 1+ \dfrac{1}{\tau} \right).
			\label{equ-hfbound}
		\end{equation}
		This is exactly the same form as \eqref{equ-hf} but with $T$ replaced by $\tau$ and an extra factor $\mathrm{e}^{\frac{2s}{n}}$ in the term $\bigl(R_2\mathrm{e}^{\frac{s+\tau}{n}}\bigr)^2$. Note that the right-hand side depends on $s$ only through this explicit exponential factor.

		We now fix an arbitrary final time $t>0$ and consider two cases to obtain the desired Harnack inequality \eqref{harnackest}. Notice that $R_2\geq R_1$ and $0<3-2\varepsilon_\tau<3$. When $0<\tau\leq1$, the extra factor $\mathrm e^{(3-2\varepsilon_\tau)\tau/n}$ is absorbed into $C(n,F)$, and $(R_2/R_1)^{3-2\varepsilon_\tau}\leq(R_2/R_1)^3$.
		
		\textit{Case 1: $0<t\leq 1$.} Applying \eqref{equ-hfbound} with $s=0$ and $\tau=t$ gives
		\begin{equation}
			\dfrac{1}{H_F^2}(p,t)\leq C(n,F)\left(\dfrac{R_2}{R_1}\right)^{3}\left( R_2 \mathrm{e}^{\frac{t}{n}} \right)^2 \left( 1+\dfrac{1}{t} \right).
			\label{equ-case1}
		\end{equation}
		For $0 < t \leq 1$, we have the elementary estimate
		\begin{equation*}
			\sqrt{1+ \dfrac{1}{t}} \leq \sqrt{\dfrac{2}{t}} \leq \sqrt{2}\left( 1 + \dfrac{1}{t^{\frac{1}{2}}} \right).
		\end{equation*}
		Taking the square root in \eqref{equ-case1} and applying the above inequality yields
		\begin{equation*}
			\frac{1}{H_F}(p,t)\leq C(n,F)\left(\frac{R_2}{R_1}\right)^{\frac{3}{2}} R_2\mathrm{e}^{\frac{t}{n}}\left(1+\frac{1}{t^{\frac{1}{2}}}\right).
		\end{equation*}
		
		\textit{Case 2: $t>1$.} Take $s=t-1>0$ and $\tau=1$. Then $s+\tau=t$ and \eqref{equ-hfbound} becomes
		\begin{equation}
			\dfrac{1}{H_F^2}(p,t)\leq C(n,F)\left(\dfrac{R_2}{R_1}\right)^{3}\left( R_2 \mathrm{e}^{\frac{t}{n}} \right)^2.
			\label{equ-case2}
		\end{equation}
		Taking the square root in \eqref{equ-case2} and noting that $1 \leq 1 + 1/\sqrt{t}$ for $t > 1$, we obtain the same inequality
		\begin{equation*}
			\frac{1}{H_F}(p,t)\leq C(n,F)\left(\frac{R_2}{R_1}\right)^{\frac{3}{2}} R_2\mathrm{e}^{\frac{t}{n}}\left(1+\frac{1}{t^{\frac{1}{2}}}\right).
		\end{equation*}
		This completes the proof of Theorem \ref{prop-harnack}.

	\section{Global smooth solutions under weak initial conditions}
	\label{sec4}
	
	In this section we prove Theorem \ref{higher-our}. 	We first construct a family of strictly $F$-mean convex hypersurfaces to approximate the initial data. These approximating hypersurfaces are then employed as initial values for the classical inverse anisotropic mean curvature flow. Finally, we pass to the limit to complete the proof of Theorem \ref{higher-our}. 
	
	\subsection{Approximation lemma}\label{sec.4-1}
	By mollifying local graphs and evolving by anisotropic mean curvature flow, we obtain smooth star-shaped hypersurfaces that approximate the original $C^1$ data while keeping $H_F$ uniformly bounded and strictly positive.

\begin{lemma}
\label{lem-approximate}
Let $F$ be a Minkowski norm, and let $M_0\subset\mathbb{R}^{n+1}$ be a closed, embedded, strictly
star-shaped hypersurface of class $C^1$ satisfying
\begin{equation*}
 0<R_1\leq \chi_0\leq R_2
\end{equation*}
for some positive constants $R_1$ and $R_2$. Assume that $M_0$ has nonnegative weak anisotropic mean curvature satisfying
\begin{equation}
\label{eq:initial-H-approx}
 0\leq H_F\leq\Lambda
 \qquad\text{for }\mathcal H^n\text{-a.e. on }M_0
\end{equation}
for some fixed constant $\Lambda>0$. Then there exist
$\varepsilon_0>0$ and a family of smooth closed hypersurfaces
$\{M_0^\varepsilon\}_{0<\varepsilon<\varepsilon_0}$ such that:
\begin{enumerate}
 \item[\textup{(a)}] $  M_0^\varepsilon\longrightarrow M_0$ as $\varepsilon\downarrow0$ in $C^1$.

 \item[\textup{(b)}]
 Each $M_0^\varepsilon$ is strictly star-shaped and, after decreasing
 $\varepsilon_0$ if necessary,
 \begin{equation}
 \label{equ-anibd}
  \frac{R_1}{2}\leq\chi_0^\varepsilon\leq2R_2,
  \qquad
  \max_{M_0^\varepsilon}F^\circ(X)\leq2R_2.
 \end{equation}

 \item[\textup{(c)}]
 The anisotropic mean curvature of $M_0^\varepsilon$ satisfies
 \begin{equation*}
  0<H_F^\varepsilon\leq C_0\Lambda,
 \end{equation*}
 where $C_0$ is independent of $\varepsilon$.
\end{enumerate}
\end{lemma}

\begin{proof}
We divide the proof into three steps.

\medskip
\noindent
\textbf{Step 1. Elliptic regularity of $M_0$.}
Since $M_0$ is $C^1$, every point has a neighborhood in which $M_0$
is represented as a graph
\begin{equation*}
 x=y^ie_i+u_0(y)e_0
\end{equation*}
over a ball $B\subset\mathbb R^n$. We use the graph orientation fixed in Subsection~\ref{appB}, namely
\begin{equation*}
 \nu=\frac{Du_0-\phi^0}{\sqrt{1+|Du_0|^2}}.
\end{equation*}
Then the weak first variation formula \eqref{equ-wamc} is equivalent to
\begin{equation*}
 \int_B D_iF(Du_0-\phi^0)D_i\zeta\,dy
 =-\int_B H_F\zeta\,dy,
 \qquad \zeta\in C_c^\infty(B).
\end{equation*}
Thus
\begin{equation}
\label{eq:divergence-graph-H}
 \operatorname{div}\big(DF(Du_0-\phi^0)\big)=H_F
\end{equation}
weakly in $B$.

After shrinking $B$ if needed, the vector $Du_0-\phi^0$ stays in a compact subset of
$\mathbb R^{n+1}\setminus\{0\}$. Hence the matrix
\begin{equation*}
 a^{ij}(Du_0):=D^iD^jF(Du_0-\phi^0)
\end{equation*}
is uniformly elliptic and uniformly continuous. In particular, for some $\lambda>0$,
\begin{equation*}
 \lambda |\xi|^2\leq a^{ij}(Du_0)\xi_i\xi_j\leq \lambda^{-1}|\xi|^2.
\end{equation*}
The local difference-quotient theory for uniformly elliptic quasilinear equations in divergence form applies to \eqref{eq:divergence-graph-H}, see \cite[Chapter IV]{LU68}. It gives $u_0\in W_{\mathrm{loc}}^{2,2}(B)$. Therefore \eqref{eq:divergence-graph-H} can be written almost everywhere as
\begin{equation*}
 a^{ij}(Du_0)D_{ij}u_0=H_F.
\end{equation*}
This is now a linear non-divergence equation with continuous uniformly elliptic coefficients. The interior $W^{2,p}$ estimate \cite[Theorem 9.11]{GT01} gives, for every $B'\Subset B$ and $1<p<\infty$,
\begin{equation*}
 \|u_0\|_{W^{2,p}(B')}
 \leq C\bigl(1+\|H_F\|_{L^p(B)}\bigr).
\end{equation*}
Here $C$ depends on $B'\Subset B$, $p$, $F$, and the $C^1$ bound of the graph. It does not depend on any later approximation. Since $H_F\in L^\infty$, the estimate holds for every finite $p$. Taking $p>n$ and using Morrey's embedding gives
\begin{equation}
\label{eq:initial-regularity}
 M_0\in C^{1,\beta}\cap W^{2,p}
 \qquad\text{for every }0<\beta<1,\quad 1\leq p<\infty.
\end{equation}

\medskip
\noindent
\textbf{Step 2. Smoothing via anisotropic mean curvature flow.}
Strict star-shapedness allows us to write $M_0$ as a radial graph $M_0=\{r_0(z)z:z\in\mathbb S^n\}$  for a positive function $r_0$. By \eqref{eq:initial-regularity}, $r_0\in C^{1,\beta}\cap W^{2,p}$ for all $0<\beta<1$ and $p<\infty$. Radial mollification gives positive functions $r_i\in C^\infty(\mathbb S^n)$ such that
\begin{equation}
\label{eq:radial-mollification}
 r_i\longrightarrow r_0
 \qquad\text{in }C^{1,\beta}\cap W^{2,p}
\end{equation}
for every $0<\beta<1$ and $p<\infty$. Set $M_i=\{r_i(z)z:z\in\mathbb S^n\}$. In every fixed graph chart,
\begin{equation*}
 H_{F,i}(\cdot,0)=a^{ab}(Du_i(\cdot,0))D_{ab}u_i(\cdot,0),
 \qquad
 a^{ab}(q)=D^aD^bF(q-\phi^0).
\end{equation*}
Consequently, after pulling the functions and area measures back to $M_0$ by the graph parametrizations,
\begin{equation}
\label{eq:initial-H-strong}
 H_{F,i}(\cdot,0)\longrightarrow H_F(\cdot,0)
 \qquad\text{strongly in }L^p(M_0)
\end{equation}
for every finite $p$.

Fix $\beta_0\in(0,1)$. Starting from $M_i$, consider anisotropic mean curvature flow
\begin{equation*}
 \partial_\varepsilon X_i=-H_{F,i}\nu_{F,i}.
\end{equation*}
We prove that the flows exist on a time interval $[0,\varepsilon_0]$ independent of $i$. Choose finitely many triples of graph cylinders $\mathcal C^0_\alpha\Subset\mathcal C^1_\alpha\Subset\mathcal C^2_\alpha$, so that the innermost cylinders cover $M_0$. By \eqref{eq:radial-mollification}, after discarding finitely many terms, all $M_i$ are graphs in the outer cylinders with a common $C^{1,\beta_0}$ bound and fixed positive graph and lateral-boundary margins. Choose $K$ strictly larger than the common initial gradient bound. Since all $M_i$ are uniformly strictly star-shaped, choose $\delta>0$ independently of $i$, so small that whenever a radial graph $r$ satisfies $\|r-r_i\|_{C^1(\mathbb S^n)}<2\delta$,  it remains strictly star-shaped and all the fixed outer graph representations persist with gradients below $K$ and with positive graph and lateral-boundary margins.

For each $i$, let $\theta_i$ be the supremum of the times for which the smooth flow exists as a radial graph and
\begin{equation*}
 \|r_i(\cdot,s)-r_i(\cdot,0)\|_{C^1(\mathbb S^n)}<2\delta.
\end{equation*}
On $[0,\theta_i)$ every local graph function satisfies
\begin{equation}
\label{eq:local-AMCF-graph}
 \partial_\varepsilon u_i
 =A^{ab}(Du_i)D_{ab}u_i,
 \qquad
 A^{ab}(p)=F(p-\phi^0)
 D^2F|_{p-\phi^0}(\phi^a,\phi^b).
\end{equation}
By \eqref{eq:graph-uniform-parabolicity}, this equation is uniformly parabolic with constants independent of $i$ on the stopping interval. 
All estimates in this paragraph are local. They are made on one fixed outer graph cylinder and recorded on the corresponding inner cylinder. We suppress the index of the cylinder. On the stopping interval the gradients stay below $K$. Hence the coefficients $A^{ab}$ in \eqref{eq:local-AMCF-graph} are uniformly parabolic. All derivatives of $A^{ab}$ are uniformly bounded on the relevant compact set of gradient values.

The common $C^{1,\beta_0}$ bound gives, at $s=0$,
\begin{equation*}
 |u_i(y,0)-\ell_{i,y_0}(y)|\leq C|y-y_0|^{1+\beta_0},
 \qquad
 \ell_{i,y_0}(y):=u_i(y_0,0)+Du_i(y_0,0)\cdot(y-y_0).
\end{equation*}
For fixed $y_0$, consider the functions
\begin{equation*}
 \ell_{i,y_0}(y)\pm C\bigl(|y-y_0|^2+\kappa s\bigr)^{(1+\beta_0)/2}.
\end{equation*}
A direct calculation uses the uniform parabolicity of \eqref{eq:local-AMCF-graph}. It shows that, after increasing $C$ and choosing $\kappa$ large, these functions are a super-solution and a sub-solution. The choice of $C$ and $\kappa$ is independent of $i$ and $y_0$. The fixed lateral margins allow the comparison principle to be used for small $s$. We get
\begin{equation*}
 \|u_i(\cdot,s)-u_i(\cdot,0)\|_{L^\infty}
 \leq C s^{(1+\beta_0)/2}.
\end{equation*}
The scaled interior H\"older-gradient estimate for uniformly parabolic quasilinear equations \cite[Theorem~12.3]{Lie96} applies with these constants. Interpolation then gives numbers $\alpha\in(0,\beta_0]$ and $\gamma_0:=\frac{\alpha\beta_0}{1+\alpha}>0$ such that
\begin{equation}
\label{eq:C1-initial-modulus}
 \|u_i(\cdot,s)-u_i(\cdot,0)\|_{C^1}
 \leq C s^{\gamma_0/2},
 \qquad 0<s<\min\{\theta_i,\varepsilon_0\}.
\end{equation}
Here and below, $C$ and $\varepsilon_0$ are independent of $i$.

We now estimate the second derivatives. Fix a point at time $s$ and take a backward cylinder $Q_\rho$ with radius $\rho=c\sqrt{s}$. The constant $c>0$ depends only on the fixed nesting of the graph cylinders. By \eqref{eq:C1-initial-modulus} and the initial $C^{1,\beta_0}$ bound, after subtracting the initial tangent plane we have
\begin{equation*}
 \operatorname*{osc}_{Q_\rho}(u_i-\ell_{i,y_0})\leq C\rho^{1+\gamma_0}.
\end{equation*}
Since $A^{ab}$ is smooth on the compact gradient range,
\begin{equation*}
 [A^{ab}(Du_i)]_{C^{\alpha,\alpha/2}(Q_\rho)}
 \leq C[Du_i]_{C^{\alpha,\alpha/2}(Q_\rho)}.
\end{equation*}
After parabolic rescaling by $\rho$ and normalization by the above oscillation, the coefficients have uniformly bounded $C^{\alpha,\alpha/2}$ norm. The interior parabolic Schauder estimate \cite[Theorem~4.9]{Lie96} gives
\begin{equation*}
 |D^2u_i(y_0,s)|\leq C\rho^{-1+\gamma_0}.
\end{equation*}
Since $\rho=c\sqrt{s}$, we obtain
\begin{equation}
\label{eq:local-D2-smoothing}
 \|D^2u_i(\cdot,s)\|_{L^\infty}
 \leq C s^{-\frac{1-\gamma_0}{2}},
 \qquad 0<s<\min\{\theta_i,\varepsilon_0\}.
\end{equation}

Because the cover is finite, \eqref{eq:C1-initial-modulus} also gives
\begin{equation*}
 \|r_i(\cdot,s)-r_i(\cdot,0)\|_{C^1(\mathbb S^n)}
 \leq C s^{\gamma_0/2}.
\end{equation*}
Choose $\varepsilon_0>0$ so small that the right-hand side is less than $\delta$ for $0<s\leq\varepsilon_0$. Thus the radial $C^1$ stopping condition cannot occur before $\varepsilon_0$. If $\theta_i<\varepsilon_0$, then \eqref{eq:local-D2-smoothing} and the standard higher-order estimates on $[\theta_i/2,\theta_i)$ give the bounds needed for continuation. The uniformly parabolic graphical flow then extends past $\theta_i$. This is a contradiction. Hence $\theta_i\geq\varepsilon_0$ for every $i$.

Since the inner cylinders cover the hypersurface, \eqref{eq:local-D2-smoothing} gives the global curvature estimate
\begin{equation}
\label{eq:curvature-smoothing}
 \|h_{i,\varepsilon}\|_{L^\infty(M_{i,\varepsilon})}
 \leq C\varepsilon^{-\frac{1-\gamma_0}{2}},
 \qquad 0<\varepsilon\leq\varepsilon_0.
\end{equation}
For every $0<\tau<\varepsilon_0$ and every $k\geq0$, repeated interior Schauder estimates \cite[Theorem~4.9]{Lie96} and differentiation of the equation give
\begin{equation}
\label{eq:AMCF-higher-uniform}
 \sup_{\tau\leq\varepsilon\leq\varepsilon_0}
 \|M_{i,\varepsilon}\|_{C^k}\leq C(k,\tau),
\end{equation}
again uniformly in $i$.

Put $ H_{F,i,-}=\max\{-H_{F,i},0\}$ and $ H_{F,i,+}=\max\{H_{F,i},0\}$. 
Equations \eqref{equ-amceq} and \eqref{equ-amcfarea}, together with \eqref{intebyparts}, imply
\begin{align*}
 \frac{d}{d\varepsilon}\int_{M_{i,\varepsilon}}H_{F,i,-}^2\,d\mu_F
 &\leq-2\int_{M_{i,\varepsilon}}
 |\nabla H_{F,i,-}|_{g}^2\,d\mu_F \notag\\
 &\qquad+2\|h_{i,\varepsilon}\|_{L^\infty}^2
 \int_{M_{i,\varepsilon}}H_{F,i,-}^2\,d\mu_F.
\end{align*}
Because the square of the right-hand side of \eqref{eq:curvature-smoothing} is integrable at zero,
\begin{equation*}
 \int_0^{\varepsilon_0}\|h_{i,s}\|_{L^\infty}^2\,ds
 \leq \frac{C}{\gamma_0}\varepsilon_0^{\gamma_0}.
\end{equation*}
Gronwall's inequality therefore gives a constant $C_*$, independent of $i$ and $\varepsilon$, such that
\begin{equation}
\label{eq:negative-part-approximants-bound}
 \|H_{F,i,-}(\cdot,\varepsilon)\|_{L^2(M_{i,\varepsilon})}
 \leq C_*\|H_{F,i,-}(\cdot,0)\|_{L^2(M_i)}.
\end{equation}
By \eqref{eq:initial-H-strong} and $H_F(\cdot,0)\geq0$, the right-hand side tends to zero as $i\to\infty$.

Similarly, for every $p\geq2$,
\begin{align*}
 \frac{d}{d\varepsilon}\int_{M_{i,\varepsilon}}H_{F,i,+}^p\,d\mu_F
 &=-p(p-1)\int_{M_{i,\varepsilon}}H_{F,i,+}^{p-2}
 |\nabla H_{F,i,+}|_{g}^2\,d\mu_F \notag\\
 &\quad+p\int_{M_{i,\varepsilon}}H_{F,i,+}^p
 |h_{i,\varepsilon}|_{g}^2\,d\mu_F
 -\int_{M_{i,\varepsilon}}H_{F,i,+}^{p+2}\,d\mu_F.
\end{align*}
After taking the $p$-th root and applying Gronwall,
\begin{equation}
\label{eq:positive-part-approximants-bound}
 \|H_{F,i,+}(\cdot,\varepsilon)\|_{L^p(M_{i,\varepsilon})}
 \leq C_*\|H_{F,i,+}(\cdot,0)\|_{L^p(M_i)},
\end{equation}
where the same enlarged constant $C_*$ is independent of $p$. The initial convergence \eqref{eq:initial-H-strong} and \eqref{eq:initial-H-approx} yield
\begin{equation}
\label{eq:initial-positive-Lp}
 \lim_{i\to\infty}\|H_{F,i,+}(\cdot,0)\|_{L^p(M_i)}
 =\|H_F(\cdot,0)\|_{L^p(M_0)}
 \leq\Lambda |M_0|_F^{1/p}.
\end{equation}

\medskip
\noindent
\textbf{Step 3. Passing to the limit.}
The estimates \eqref{eq:AMCF-higher-uniform} give, after a diagonal subsequence, smooth convergence
\begin{equation*}
 M_{i,\varepsilon}\longrightarrow M_\varepsilon
 \qquad\text{on every compact subinterval of }(0,\varepsilon_0].
\end{equation*}
The limit solves anisotropic mean curvature flow. The initial-time modulus \eqref{eq:C1-initial-modulus}, together with $M_i\to M_0$ in $C^1$, passes to the limit and gives
\begin{equation}
\label{eq:AMCF-C1-trace}
 M_\varepsilon\longrightarrow M_0
 \qquad\text{in }C^1\quad\text{as }\varepsilon\downarrow0.
\end{equation}
The curvature estimate also passes to the limit:
\begin{equation*}
 \|h_\varepsilon\|_{L^\infty(M_\varepsilon)}^2
 \leq C\varepsilon^{-1+\gamma_0}.
\end{equation*}

Letting $i\to\infty$ in \eqref{eq:negative-part-approximants-bound} gives
\begin{equation}
\label{eq:H-nonnegative}
 H_F(\cdot,\varepsilon)\geq0
 \qquad\text{on }M_\varepsilon.
\end{equation}
Letting $i\to\infty$ in \eqref{eq:positive-part-approximants-bound} and using \eqref{eq:initial-positive-Lp}, we find, for every $p\geq2$,
\begin{equation*}
 \|H_F(\cdot,\varepsilon)\|_{L^p(M_\varepsilon,d\mu_F)}
 \leq C_*\Lambda |M_0|_F^{1/p}.
\end{equation*}
The constant is independent of $p$, so letting $p\to\infty$ gives $\sup_{M_\varepsilon}H_F\leq C_*\Lambda$. 

The inequality in \eqref{eq:H-nonnegative} is strict for positive time. Indeed, a smooth closed hypersurface cannot have identically zero anisotropic mean curvature: at a point at which $F^\circ(X)$ is maximal, comparison with the tangent Wulff shape gives
\begin{equation*}
 \max_{M_\varepsilon}H_F
 \geq\frac{n}{\max_{M_\varepsilon}F^\circ(X)}>0.
\end{equation*}
Thus $H_F(\cdot,\varepsilon/2)$ is nonnegative and not identically zero. The strong maximum principle applied to \eqref{equ-amceq} on $[\varepsilon/2,\varepsilon]$ yields
\begin{equation*}
 H_F(\cdot,\varepsilon)>0
 \qquad\text{on }M_\varepsilon.
\end{equation*}

Finally, \eqref{eq:AMCF-C1-trace} and the continuity of the anisotropic support function imply the estimates in \eqref{equ-anibd} for small $\varepsilon$. Setting $M_0^\varepsilon:=M_\varepsilon$ and $C_0:=C_*$ completes the proof.
\end{proof}


	\subsection{From approximants to a global smooth solution} 

\begin{proof}[Proof of Theorem \ref{higher-our}]
Under the assumption,  choose a constant $\Lambda>0$ such that
		\begin{equation*}
			0\leq H_F \leq \Lambda \qquad \text{a.e.\ on}\ M_0.
		\end{equation*}
		By Lemma \ref{lem-approximate}, there exists a sequence of smooth, strictly star-shaped, strictly $F$-mean convex closed hypersurfaces $M_0^{\varepsilon}$ for $\varepsilon>0$ small satisfying
		\begin{equation*}
			\dfrac{R_1}{2} \leq \chi^{\varepsilon}_0 \leq \max\limits_{M_0^{\varepsilon}}F^{\circ}(X) \leq 2R_2
		\end{equation*}
		and $0< H_F^{\varepsilon} \leq C_0\Lambda$,  where $C_0$ is independent of $\varepsilon$. Fix each $\varepsilon>0$, let $X^{\varepsilon}:M\times(0,\infty)\to\mathbb R^{n+1}$
		be the smooth solution to IAMCF \eqref{IAMCF} with initial hypersurface $M_0^{\varepsilon}$, and set $M_t^{\varepsilon}=X^{\varepsilon}(M,t)$. 
        
         By \cite[Theorem 1.1]{Xia17}, the hypersurface $M_t^{\varepsilon}$ is strictly star-shaped, strictly $F$-mean convex and exists for all $t>0$. Moreover, Lemma \ref{lem-anisosupp} implies that the anisotropic support function $\chi^{\varepsilon}(p,t)$ of $M_t^{\varepsilon}$ satisfies
		\begin{equation}
			\label{equ-suppest}
			\dfrac{R_1}{2}\mathrm{e}^{\frac{t}{n}} \leq \chi^{\varepsilon}\leq F^{\circ}(X^{\varepsilon})\leq 2R_2 \mathrm{e}^{\frac{t}{n}}.
		\end{equation}
		The Harnack estimate in Theorem \ref{prop-harnack} then yields
		\begin{equation}
			\label{equ-harnackseq}
			\dfrac{1}{H_F^{\varepsilon}(p,t)} \leq C(n,F)\left(\dfrac{R_2}{R_1}\right)^{\frac{3}{2}}\left(1+\frac{1}{t^{\frac{1}{2}}} \right)R_2\mathrm{e}^{\frac{t}{n}}.
		\end{equation}
		Consequently, for every $0<\tau<T<\infty$ there exists a constant $h_{-}(\tau,T)>0$, independent of $\varepsilon$, such that
		\begin{equation}
			\label{equ-hl}
			H_F^{\varepsilon}(p,t)\geq h_-(\tau,T)\ \ \text{on}\ \ M\times [\tau,T].
		\end{equation}
		On the other hand, applying the maximum principle to the evolution equation \eqref{equ-amc} of $H_F$ gives, for any $t>0$,
		\begin{equation}
			\sup_{M_t^\varepsilon}H_F^\varepsilon \leq \sup_{M_0^\varepsilon} H_F^{\varepsilon}  \leq C_0\Lambda.
			\label{equ-h+}
		\end{equation}
		Both estimates \eqref{equ-hl} and \eqref{equ-h+} are uniform  in $\varepsilon$.
		
		 Since $M_t^{\varepsilon}$ remains strictly star-shaped, it can be written as a radial graph:
		\begin{equation*}
			X^{\varepsilon}(z,t)=r^{\varepsilon}(z,t)z, \qquad z\in\mathbb S^n.
		\end{equation*}
		Let $\overline{g}=(\sigma_{ij})$ and $\overline{\nabla}$ be the standard metric and connection on $\mathbb{S}^n$, and set
		\begin{equation*}
			v^{\varepsilon}=\sqrt{1+|\overline{\nabla}\log r^{\varepsilon}|_{\overline{g}}^2}.
		\end{equation*}
		Then
		\begin{equation*}
			\chi^{\varepsilon}=\frac{r^{\varepsilon}}{v^{\varepsilon}F(\nu^{\varepsilon})},\qquad \nu^{\varepsilon}=\dfrac{z-\overline{\nabla}\log r^{\varepsilon}}{v^{\varepsilon}}.
		\end{equation*}
		Moreover,
		\begin{equation}\label{s4.HF}
			H_F^{\varepsilon} = A_{ij}(\nu^{\varepsilon})\dfrac{1}{r^{\varepsilon}v^{\varepsilon}}\left[ \delta_{ij} - \left(\sigma^{ik}-\dfrac{\overline{\nabla}^i r^{\varepsilon}\overline{\nabla}^k r^{\varepsilon}}{(r^{\varepsilon}v^{\varepsilon})^2} \right)\left( \dfrac{\overline{\nabla}_k\overline{\nabla}_j r^{\varepsilon}}{r^{\varepsilon}} - \dfrac{\overline{\nabla}_k r^{\varepsilon} \overline{\nabla}_j r^{\varepsilon}}{(r^{\varepsilon})^2} \right) \right],
		\end{equation}
		where
		\begin{equation*}
			A_{ij}(\nu^{\varepsilon}) = \overline{\nabla}_i\overline{\nabla}_j F(\nu^{\varepsilon})+F(\nu^{\varepsilon})\sigma_{ij}.
		\end{equation*}
		The flow equation \eqref{IAMCF} then becomes
		\begin{equation}
			\label{equ-r}
			\partial_t r^{\varepsilon}=\frac{v^{\varepsilon} F(\nu^{\varepsilon})}{H_F^{\varepsilon}}.
		\end{equation}
		These are formulas derived in \cite[Section 4]{Xia17}.

The scalar equation \eqref{equ-r} for $r^{\varepsilon}$ is a fully nonlinear parabolic equation on $\mathbb{S}^n\times[\tau,T]$. The uniform estimate \eqref{equ-suppest} provides uniform $C^{1}$ bounds for the graph function $r^{\varepsilon}$, and \eqref{equ-hl} - \eqref{equ-h+} provide uniform two-sided bounds for $H_F^{\varepsilon}$. The linearization of \eqref{equ-r} with respect to $\overline\nabla_i\overline\nabla_j r^\varepsilon$ is uniformly parabolic. In local coordinates it is
\begin{align*}
			\dfrac{\partial}{\partial \overline{\nabla}_i\overline{\nabla}_j r^{\varepsilon}}\left(\frac{v^{\varepsilon}F(\nu^{\varepsilon})}{H_F^{\varepsilon}}\right) =& -\dfrac{v^{\varepsilon}F(\nu^{\varepsilon})}{\left(H_F^{\varepsilon}\right)^2}\dfrac{\partial H_F^{\varepsilon}}{\partial \overline{\nabla}_i\overline{\nabla}_j r^{\varepsilon}} \\
			=& \dfrac{F(\nu^{\varepsilon})}{\left(r^{\varepsilon}H_F^{\varepsilon}\right)^2}A_{ik}(\nu^{\varepsilon})\left(\sigma^{kj}-\dfrac{\overline{\nabla}^k r^{\varepsilon}\overline{\nabla}^j r^{\varepsilon}}{(r^{\varepsilon}v^{\varepsilon})^2} \right).
		\end{align*}
While the flow equation \eqref{equ-r} is fully nonlinear, the spatial operator \eqref{s4.HF} defining $H_F^{\varepsilon}$ at a fixed time is quasilinear elliptic. The $C^{2,\alpha}$ estimate in \cite[Pages 119--122]{Xia17} applies on every slab $[\tau,T]$ and its constants depend only on the uniform bounds for $r^{\varepsilon}$, $|\overline{\nabla}r^{\varepsilon}|_{\overline{g}}$ and $H_F^{\varepsilon}$. Parabolic Schauder estimates give the $C^{2+\alpha,1+\alpha/2}$ bound. Differentiating the equation and iterating the Schauder estimates gives, for every $0<\tau<T<\infty$, $k\geq 0$, and $0<\alpha<1$,
\begin{equation}
	\label{equ-high}
	\|r^{\varepsilon}\|_{C^{k+2+\alpha,(k+2+\alpha)/2}(\mathbb{S}^n \times[\tau,T])}\leq C(k,\tau,T,\alpha,n,F,R_1,R_2,\Lambda,M_0)
\end{equation}
uniformly in $\varepsilon$.

		Now the family $\{r^{\varepsilon}\}$ is precompact in $C^\infty_{\mathrm{loc}}(\mathbb S^n\times(0,\infty))$. By the Arzel\`a--Ascoli theorem and a diagonal argument, there exists a subsequence converging smoothly to
		\begin{equation*}
			r:\mathbb S^n\times(0,\infty)\to(0,\infty),
		\end{equation*}
		which solves \eqref{equ-r}. Hence $M_t=\{r(z,t)z:z\in\mathbb S^n\}$ is a smooth solution of IAMCF \eqref{IAMCF} on $(0,\infty)$. In particular, this solution is strictly star-shaped and satisfies the Harnack estimate \eqref{harnackest} by letting $\varepsilon \to 0$ in estimates \eqref{equ-suppest} and \eqref{equ-harnackseq}.

		 To show that $M_t$ converges uniformly to $M_0$ as $t\to 0$, by \eqref{equ-r}, \eqref{equ-high} and \eqref{equ-harnackseq},
		\begin{equation*}
			|\partial_t r^{\varepsilon}|\leq C(1+t^{-1/2})e^{t/n},
		\end{equation*}
		where $C$ is independent of $\varepsilon$. Integrating in time yields
		\begin{equation}
			\label{equ-rinfty}
			\|r^{\varepsilon}(\cdot,t)-r^{\varepsilon}(\cdot,0)\|_{L^\infty(\mathbb{S}^n)}\leq \int_0^t |\partial_s r^{\varepsilon}|\mathrm{d}s \leq C(t+t^{1/2})
		\end{equation}
		for $t<1$. For every fixed $t>0$, the subsequential convergence gives $r^{\varepsilon}(\cdot,t)\to r(\cdot,t)$ uniformly on $\mathbb S^n$. At the initial time, Lemma~\ref{lem-approximate} gives $r^{\varepsilon}(\cdot,0)\to r_0$ uniformly. Letting $\varepsilon\to 0$ in \eqref{equ-rinfty} gives
		\begin{equation*}
			\|r(\cdot,t)-r_0\|_{L^\infty(\mathbb S^n)}\to 0 \qquad\text{as}\ t \downarrow 0.
		\end{equation*}
		Therefore $M_t\to M_0$ uniformly as $t\downarrow 0$. This completes the proof.
        \end{proof}

We conclude this section with the following proposition, which says that the smooth solution $M_t$ obtained in Theorem \ref{higher-our} is also a weak solution starting from $M_0$. 
\begin{proposition}
	\label{prop:smooth-continuation-weak}
	Let $\{M_t\}_{t>0}$ be the smooth solution obtained in Theorem \ref{higher-our}, and let $E_t$ be the bounded region enclosed by $M_t$. Define the arrival-time function $u$ by $u=t$ on $M_t$, and extend it inside $E_0$ so that $u<0$ there. Then $\{E_t\}_{t>0}$ is a weak solution of IAMCF starting from $E_0$ in the sense of Definition \ref{def-minimizes}.
\end{proposition}

\begin{proof}
	The curvature upper bound \eqref{equ-h+} passes to the limit and gives $H_F\leq C_0\Lambda$ along the smooth flow. Since a smooth IAMCF written as a level-set flow satisfies $F(Du)=H_F$ on its regular level sets, the function $u$ is locally Lipschitz in $\mathbb{R}^{n+1}\setminus E_0$.

	Fix $t>0$, let $G$ be a competitor with $E_t\triangle G\Subset\mathbb{R}^{n+1}\setminus E_0$, and choose a compact set $K\Subset\mathbb{R}^{n+1}\setminus E_0$ containing $E_t\triangle G$. The support estimate \eqref{equ-suppest} passes to the limiting flow and gives
	\begin{equation*}
		F^\circ(X)\geq\chi\geq \frac{R_1}{2}e^{s/n}
		\qquad\text{on }M_s.
	\end{equation*}
	Since the hypersurfaces are radial graphs, this implies
	\begin{equation*}
		\mathcal W_{(R_1/2)e^{s/n}}(0)\subset E_s.
	\end{equation*}
	Thus the regions $E_s$ exhaust $\mathbb R^{n+1}$, and we may choose $T>t$ so large that $K\Subset E_T$. On the other hand, the uniform convergence $E_\delta\to E_0$ as $\delta\downarrow0$ allows us to choose $0<\delta<t$ with $K\Subset\mathbb R^{n+1}\setminus\overline{E_\delta}$. Hence
	\begin{equation*}
		K\Subset E_T\setminus\overline{E_\delta}.
	\end{equation*}
	The family $\{M_s\}_{\delta\leq s<T}$ is smooth and has positive anisotropic mean curvature. Lemma \ref{lem-smweak} therefore applies in the entire domain containing the variation and gives
	\begin{equation*}
		J_{F,u}^K(E_t)\leq J_{F,u}^K(G).
	\end{equation*}
	This is the weak minimizing property outside $E_0$.
\end{proof}

	\section{Higher regularity of the weak IAMCF}
	\label{sec5}
	In this section, we prove Theorem \ref{thm-main}. We first need to show that there exists a sufficiently large time $t_0$ such that the weak solution $M_{t_0}$ is $C^1$ and strictly star-shaped. This relies on compactness for anisotropic almost minimizers \cite{DPM15}, the interior small-excess regularity theorem of Duzaar--Steffen \cite{DS02}, and the asymptotically Wulff shape result in Theorem~\ref{thm:anisotropic-asymptotic}. In the isotropic theory of Huisken--Ilmanen \cite{HI08}, the corresponding large-time regularity follows from the special asymptotic regularity of the Euclidean weak IMCF. In the anisotropic setting, the available asymptotic behavior in Theorem~\ref{thm:anisotropic-asymptotic} gives only Hausdorff convergence to an expanding Wulff shape. We therefore recover the required $C^{1,\beta}$ regularity by using the variational almost-minimality of the level sets together with anisotropic small-excess regularity.

\begin{proposition}
\label{prop:C1-blowdown}
Let $u$ be the proper weak solution of \eqref{eq:WIAMCF} starting from a
bounded open set $\Omega\subset\mathbb R^{n+1}$ with smooth boundary.
Set $E_t=\{u<t\}$ and $M_t=\partial E_t$. Let
\begin{equation*}
 \gamma=\log\left(\frac{|\partial\mathcal W|_F}{|\partial\Omega^*|_F}\right),
 \qquad
 \rho_\infty=\exp\left(-\frac{\gamma}{n}\right),
\end{equation*}
where $\Omega^*$ is the strictly outward $F$-minimizing hull of $\Omega$.
Then there exists $\alpha\in(0,1)$ such that, for every $0<\beta<\alpha$,
\begin{equation}
\label{eq:C1-blowdown}
 e^{-\frac{t}{n}}M_t
 \longrightarrow \rho_\infty\partial\mathcal W
 \qquad\text{in }C^{1,\beta}
\end{equation}
as $t\to\infty$. In particular, for all sufficiently large $t$, the
level set $M_t$ has no singular points, is of class $C^{1,\beta}$, and
is strictly star-shaped.
\end{proposition}

\begin{proof}
Up to a translation, we may assume $0\in\Omega$. By
Theorem~\ref{thm:anisotropic-asymptotic}, we have
\begin{equation}
\label{eq:u-asymptotic-C1}
 u(x)=n\log F^\circ(x)+\gamma+o(1)
 \qquad\text{as }F^\circ(x)\longrightarrow\infty.
\end{equation}
By Lemma~\ref{lem:weak-gradient}, there exists a constant
$C=C(n,F,\Omega)>0$ such that
\begin{equation}
\label{eq:scale-gradient-C1}
 F(Du)(x)\leq \frac{C}{F^\circ(x)}
\end{equation}
for a.e. $x\in\mathbb R^{n+1}\setminus\Omega$.

For $s\geq1$, define
\begin{equation*}
 u_s(x):=u(sx)-n\log s,
 \qquad
 E_s:=\{u_s<0\}=s^{-1}E_{n\log s}.
\end{equation*}
It follows from \eqref{eq:u-asymptotic-C1} that
\begin{equation}
\label{eq:us-convergence}
 u_s\longrightarrow v:=n\log F^\circ+\gamma
 \qquad\text{locally uniformly in }
 \mathbb{R}^{n+1}\setminus\{0\}.
\end{equation}
Consequently,
\begin{equation}
\label{eq:set-convergence}
 \mathbf{1}_{E_s}\longrightarrow\mathbf{1}_{E_\infty}
 \quad\text{in }L^1_{\mathrm{loc}}
 (\mathbb{R}^{n+1}\setminus\{0\}),
 \qquad
 E_\infty=\{F^\circ<\rho_\infty\}.
\end{equation}
Since $Dv$ does not vanish on $\partial E_\infty$, the local uniform
convergence in \eqref{eq:us-convergence} also gives
\begin{equation}
\label{eq:hausdorff-convergence}
 \partial E_s\longrightarrow
 \partial E_\infty=\rho_\infty\partial\mathcal W
 \qquad\text{locally in the Hausdorff distance}.
\end{equation}
We now prove the $C^{1,\beta}$ convergence. Fix annuli
$A'\Subset A\Subset\mathbb{R}^{n+1}\setminus\{0\}$ with
$\rho_\infty\partial\mathcal W\Subset A'$. By
\eqref{eq:hausdorff-convergence}, $\partial E_s\cap A'$ lies in $A$ for all
large $s$. Also $A\subset\mathbb{R}^{n+1}\setminus s^{-1}\overline{\Omega}$
for all large $s$. The scaling invariance of the weak formulation shows that
$E_s$ minimizes $J_{F,u_s}$ with respect to compactly supported variations in
$A$.

The gradient bound scales in a simple way. By \eqref{eq:scale-gradient-C1},
\begin{equation*}
 F(Du_s)(x)=sF(Du)(sx)\leq \frac{C}{F^\circ(x)}.
\end{equation*}
Thus $\sup_A F(Du_s)\leq\Lambda_A$ with $\Lambda_A$ independent of $s$.
If $E_s\triangle G\Subset B_r(x)\Subset A$, the minimizing property gives
\begin{equation}
\label{eq:almost-minimality-Es}
 P_F(E_s;B_r(x))
 \leq P_F(G;B_r(x))+\Lambda_A|E_s\triangle G|.
\end{equation}
Hence the normalized representatives of $E_s$ are uniform
$(\Lambda_A,r_A)$-minimizers of the autonomous elliptic integrand
$\Phi(\nu)=F(\nu)$ on $A'$, for some $r_A>0$ independent of $s$. This is in
the sense of \cite[Definition~1.8 and Lemma~2.16]{DPM15}.

Let $s_j\to\infty$. By the compactness theorem
\cite[Theorem~2.9]{DPM15}, a subsequence has convergent oriented perimeter
measures. Their total variations also converge. The $L^1_{\mathrm{loc}}$
convergence in \eqref{eq:set-convergence} identifies the limit set as
$E_\infty$. Since $\partial E_\infty=\rho_\infty\partial\mathcal W$ is a
smooth hypersurface, the limit measure is the multiplicity-one perimeter
measure of $E_\infty$. Thus there is no loss of mass and no extra sheet in
the limit.

We choose finitely many balls $B_{2r_a}(x_a)\Subset A'$ centered on
$\partial E_\infty$. The radii are chosen small. First, the excess of the
smooth hypersurface $\partial E_\infty$ in each ball is below the small
constant in the regularity theorem. Second, $\Lambda_A r_a$ is below the
same threshold. By \cite[Remark~3.6]{DPM15} and the Hausdorff convergence
\eqref{eq:hausdorff-convergence}, the excess of $\partial E_{s_j}$ in
$B_{r_a}(x_a)$ is also small for all large $j$. After rescaling, the
almost-minimizing modulus in \eqref{eq:almost-minimality-Es} is bounded by
$C\Lambda_A r_a$, as in \cite[(1.16)--(1.17)]{DS02}. The interior
$\varepsilon$-regularity theorem \cite[Theorem~6.1]{DS02} then gives a
local graph representation of $\partial E_{s_j}$ in each ball. The graph
norms are bounded uniformly in $C^{1,\alpha}$ for some $\alpha\in(0,1)$.

The Hausdorff convergence fixes the sheet of each local graph. It also shows
that these graphs cover the whole boundary for large $j$. On overlaps the
graphs describe the same reduced boundary, so they patch to a normal graph
over $\rho_\infty\partial\mathcal W$. The graph function tends to zero in
$C^0$. The uniform $C^{1,\alpha}$ bound then gives convergence to zero in
$C^{1,\beta}$ for every $0<\beta<\alpha$. Since the sequence $s_j$ was
arbitrary and the limit is unique, the whole family converges. Thus
$\partial E_s$ has no singular points for all sufficiently large $s$.
Taking $s=e^{t/n}$ proves \eqref{eq:C1-blowdown}.

It remains to verify strict star-shapedness. The anisotropic support
function of the limiting Wulff shape is constant:
\begin{equation*}
 \frac{\langle X,\nu\rangle}{F(\nu)}=\rho_\infty
 \qquad\text{on }\rho_\infty\partial\mathcal W.
\end{equation*}
The anisotropic support function depends continuously on the position
and Euclidean unit normal in the $C^1$ topology. Hence
\begin{equation*}
 \frac{\langle \widetilde X,\widetilde\nu\rangle}
 {F(\widetilde\nu)}\geq\frac{\rho_\infty}{2}>0
 \qquad\text{on }e^{-t/n}M_t
\end{equation*}
for all sufficiently large $t$. Scaling back yields
\begin{equation*}
 \chi=\frac{\langle X,\nu\rangle}{F(\nu)}
 \geq\frac{\rho_\infty}{2}e^{t/n}>0
 \qquad\text{on }M_t.
\end{equation*}
This proves the final assertion.
\end{proof}

\begin{proof}[Proof of Theorem \ref{thm-main}]
After translating the coordinates, we may assume $0\in\Omega$. By
Proposition~\ref{prop:C1-blowdown}, there is a time $T_*>0$ such that
$M_t$ is a $C^{1,\beta}$ strictly star-shaped hypersurface for every
$t\geq T_*$. Let $\mathcal S$ be the full-measure set of times for which
\eqref{lem-level} holds. Choose $t_0\in \mathcal S\cap(T_*,\infty)$.
Then $M_{t_0}$ is $C^{1,\beta}$ and strictly star-shaped. Moreover, by
\eqref{lem-level}, its weak anisotropic mean curvature satisfies
\begin{equation*}
 H_F=F(Du)
 \qquad\mathcal H^n\text{-a.e. on }M_{t_0}.
\end{equation*}
Since $M_{t_0}$ is compact and lies outside the bounded set $\Omega$,
\begin{equation*}
 m_{t_0}:=\inf_{x\in M_{t_0}}F^\circ(x)>0.
\end{equation*}
Hence the gradient estimate \eqref{eq:weak-gradient} gives
\begin{equation*}
 0\leq H_F\leq C m_{t_0}^{-1}
 \qquad\mathcal H^n\text{-a.e. on }M_{t_0}.
\end{equation*}
Thus $M_{t_0}$ satisfies all hypotheses of Theorem~\ref{higher-our}.

Use $M_{t_0}$ as initial data in Theorem \ref{higher-our}. After resetting the time variable, we obtain a smooth solution $\{\widetilde M_s\}_{s>0}$ with strictly positive anisotropic mean curvature. Let $\widetilde E_s$ be the enclosed regions. Define the shifted family
	\begin{equation*}
	 \widehat E_{t_0+s}:=\widetilde E_s,
	 \qquad
	 \widehat M_{t_0+s}:=\widetilde M_s.
	\end{equation*}
	By Proposition \ref{prop:smooth-continuation-weak}, $\{\widehat E_{t_0+s}\}_{s>0}$ is a weak solution starting from $E_{t_0}$ in $\mathbb R^{n+1}\setminus E_{t_0}$. The original shifted family $\{E_{t_0+s}\}_{s>0}$ is also a weak solution in the same region. The level sets are compact, so Lemma \ref{lem-weakcom} gives uniqueness. Hence
	\begin{equation*}
	 \widehat M_t=M_t
	 \qquad\text{for all }t>t_0.
	\end{equation*}
	Therefore the original weak solution is smooth for all $t>t_0$. Taking $K=\overline{E_{t_0}}$ finishes the proof. 
\end{proof}

\section*{Acknowledgements}
This work was supported by the National Key Research and Development Program of China (2021YFA1001800), the National Natural Science Foundation of China (12531002), and the Fundamental Research Funds for the Central Universities. The third author was also supported by the China Postdoctoral Science Foundation (2025M783146).

\end{document}